\documentclass[letterpaper,final,10pt]{amsart}

\usepackage{inputenc}

\usepackage{amsmath}
\usepackage{amsfonts}
\usepackage{amssymb}
\usepackage{graphicx}
\usepackage{mathrsfs}
\usepackage[obeyspaces]{url}
\usepackage[all,cmtip,curve,knot]{xy}
\usepackage{tikz}
\usetikzlibrary{arrows}
\tikzstyle{block}=[draw opacity=0.7,line width=1.4cm]
\newtheorem{theorem}{Theorem}

\newtheorem{corollary}{Corollary}

\newtheorem{definition}{Definition}
\newtheorem{example}{Example}

\newtheorem{lemma}{Lemma}

\newtheorem{proposition}{Proposition}
\newtheorem{remark}{Remark}

\numberwithin{equation}{section}
\numberwithin{theorem}{section}
\numberwithin{lemma}{section}
\numberwithin{corollary}{section}
\numberwithin{definition}{section}
\numberwithin{example}{section}
\numberwithin{remark}{section}
\numberwithin{property}{section}
\numberwithin{proposition}{section}

\newcommand{\NS}[3][1]{#2_#1,\ldots,#2_#3}

\newcommand{\lift}[3][\pi]{\xymatrix{#2 \ar@{~>}[r]_{#1} & #3}}

\newcommand{\Minf}{\mathcal{M}_\infty}

\newcommand{\diag}[2][]{\mathrm{diag}_{#1}\left[ #2 \right]}
\newcommand{\N}[1]{\mathcal{N}(#1)}
\newcommand{\U}[1]{\mathbb{U}(#1)}

\newcommand{\I}{\mathbf{1}}

\newcommand{\RR}{\mathbb{R}}
\newcommand{\CC}{\mathbb{C}}
\newcommand{\ZZ}{\mathbb{Z}}

\newcommand{\Rep}[2]{\mathrm{C^\ast Rep}(#1,#2)}

\newcommand{\TT}[1][1]{{\mathbb{T}^{#1}}}

\newcommand{\disk}[1][2]{{\mathbb{D}^{#1}}}

\begin{document}
\title[Connecting Commuting Normal Matrices]{Connecting Commuting Normal Matrices}
\author{Fredy Vides}
\address{Department of Applied Mathematics, School of Mathematics, Faculty of Science, 
Universidad Nacional Aut\'onoma de Honduras, Ciudad Universitaria, Tegucigalpa, Honduras.}

\email{fredy.vides@unah.edu.hn}



\date{\today}

\begin{abstract}
In this document we study the local path connectivity of sets of $m$-tuples of commuting normal matrices with some additional geometric constraints in their joint spectra. In particular, given $\varepsilon>0$ and any fixed but arbitrary $m$-tuple $\mathbf{X}\in {M_n(\mathbb{C})}^m$ in the set of $m$-tuples of pairwise commuting normal matrix contractions, we prove the existence of paths between arbitrary $m$-tuples in the intersection of the previously mentioned sets of $m$-tuples in ${M_n(\mathbb{C})}^m$ and the $\delta$-ball $B_\eth(\mathbf{X},\delta)$ centered at $\mathbf{X}$ for some $\delta>0$, with respect to some suitable metric $\eth$ in ${M_n(\mathbb{C})}^m$ induced by the operator norm. Two of the key features of these matrix paths is that $\delta$ can be chosen independent of $n$, and that the paths stay in the intersection of $B_\eth(\mathbf{X},\varepsilon)$, and the set pairwise commuting normal matrix contractions with some special geometric structure on their joint spectra.

We apply these results to study the local connectivity properties of matrix $\ast$-representations of some universal commutative $C^\ast$-algebras. Some connections with the local connectivity properties of completely positive linear maps on matrix algebras are studied as well.
\end{abstract}

\maketitle

\section{Introduction}
\label{intro}

In this document we study the local path connectivity of sets of $m$-tuples of commuting normal matrices, whose joint spectra may consist of points in some arbitrary element of the family of compact differentiable manifolds contained in $\mathbb{R}^m$. In particular, given $\varepsilon>0$ and any fixed but arbitrary $m$-tuple $\mathbf{X}\in {M_n(\mathbb{C})}^m$ in the set of $m$-tuples of pairwise commuting normal matrix contractions (where matrix contraction means a matrix $X\in M_n(\mathbb{C})$ such that $\|X\| \leq 1$), we prove the existence of paths between arbitrary $m$-tuples in the intersection of the previously mentioned sets of $m$-tuples in $M_n(\mathbb{C})^m$ and the $\delta$-ball $B_\eth(\mathbf{X},\delta)$ centered at $\mathbf{X}$, with respect to some suitable metric $\eth$ in ${M_n(\mathbb{C})}^m$ induced by the operator norm (that will be defined further below). Two of the key features of these matrix paths is that $\delta$ can be chosen independent of $n$, and that the paths stay in the intersection of $B_\eth(\mathbf{X},\varepsilon)$, and the set pairwise commuting normal matrix contractions with some special geometric structure on their joint spectra.

As mentioned in the previous paragraph, we start by studying local path connectivity properties of some particular varieties, these varieties consist of $m$-tuples of pairwise commuting normal matrix contractions with some geometric constraints on their joint spectra. Given $r>0$ and a $m$-tuple of $n\times n$ matrices $\mathbf{X}\in M_n(\mathbb{C})^m$, let us write $B_\eth(\mathbf{X},r)$ to denote the set $\{\mathbf{Y}\in (M_n)^m | \eth(\mathbf{X},\mathbf{Y})\leq r\}$. We call $B_\eth(\mathbf{X},r)$ a $r$-ball centered at $\mathbf{X}$ or a $r$-neighborhood of $\mathbf{X}$. Given $\varepsilon>0$, we will prove that there is $\delta>0$ such that for any $m$-tuple $\mathbf{X}:=(X_1,\ldots,X_m)\in M_n(\mathbb{C})^m$ in some subset $Z^m$ of the set of commuting $m$-tuples of normal contractions, and any $\tilde{\mathbf{X}}:=(\tilde{X}_1,\ldots,\tilde{X}_m)\in B_\eth(\mathbf{X},\delta)\cap Z^m$, we can find  a $m$-tuple of matrix paths $\gamma=(\gamma_1,\ldots,\gamma_m)\in C([0,1],B_\eth(\mathbf{X},\varepsilon))\cap Z^m$, such that $\gamma(0)=\mathbf{X}$, $\gamma(1)=\mathbf{\tilde{X}}$, and where $\gamma(t)\in Z^m$ for each $0\leq t\leq 1$. As mentioned before, $\delta$ does not depend on $n$.

The problems which motivated the research reported in this document have a topological nature. We started this study in \cite{Vides_homotopies}, where we found some connections between the map $\Psi:=\mathrm{Ad}[W]$ described by Lemma \ref{Joint_spectral_variation_inequality_2} in \S\ref{notation} and the local homeomorphims that play a key role in the Kirby Torus Trick (introduced by R. Kirby in \cite{Torus_Trick_Kirby}) by moving sets of points in homeomorphic copies $\mathcal{T}^m$ of $\TT[m]$ {\bf \em just a little bit}, this analogy together with L.\ref{existence_of_almost_unit} provides a connection with {\bf \em topologically controlled linear algebra} (in the sense of \cite{CLA_Freedman}), which can be roughly described as the study of the relations between matrix sets and 
differentiable manifolds, more specifically, the study of the connections between approximation of functions on differential manifolds and approximation techniques for families of matrices which satisfy some special constraints, like commutativity, hermiticity and spectral behavior among others. The search for the previously mentioned analogies and connections was motivated by a question raised by M. H. Freedman regarding to the role played by the Kirby Torus Trick in linear algebra and matrix approximation theory. In section \S\ref{main_results} we obtain uniform versions of the previously mentioned analogies based on Lemma \ref{existence_jpia}.

The results presented in section \S\ref{main_results} can be used to compute, and to study the approximate and exact simultaneous block diagonalization of matrix $m$-tuples in the sense of \cite{sim_block_diag,ASD_matrices}, problems of this type appear in biomathematics, image processing and applied spectral graph theory. Another important source of motivation and inspiration for this paper came from physics, more specifically from problems related to the classification of quantum phases in the sense of \cite{connectivity_CP_maps_q_phases}.

Building on some of the ideas developed by M. A. Rieffel in \cite{Finite_groups_Rieffel,Vector_bundles_Rieffel}, by O. Bratteli, G. A. Elliott, D. E. Evans and A. Kishimoto in \cite{Basic_homotopies}, by H. Lin in 
\cite{Lin_Theorem}, by D. Hadwin in \cite{Algebraic_lifting_Hadwin}, by K. Davidson in \cite{Davidson_1} and by P. Friis and M. R{\"o}rdam in \cite{Rordam_Lin_Thm}, we proved Lemma \ref{existence_of_almost_unit}, Lemma \ref{Existence_of_1D_PMA} and Lemma \ref{Existence_of_mD_PMA}, these results together with Lemma \ref{Joint_spectral_variation_inequality_2} provide us with the matrix approximation technology that we use to prove the main results.  The main results, consisting of Theorem \ref{main_result}, Corollary \ref{main_corollary}, Theorem \ref{main_generic_result} are presented in \S\ref{main_results}. Some applications of the results in \S\ref{main_results} to the study of the local connectivity properties of matrix representations of universal commutative $C^\ast$-algebras are presented in \S\ref{applications} and some future directions are outlined in \S\ref{hints}.

\section{Preliminaries and Notation}
\label{notation}
We will write $M_{m,n}$ to denote the set $M_{m,n}(\mathbb{C})$ of $m\times n$ complex matrices, if $m=n$ we will write $M_n$, we write $M_n^m$ to denote the set of $m$-tuples of $n\times n$ complex matrices. The symbols $\mathbf{1}_n$ and $\mathbf{0}_{m,n}$ will be used to denote the identity matrix and the zero matrix in $M_n$ and $M_{m,n}$ respectively, if $m=n$ we will write $\mathbf{0}_{n}$. Given a matrix $A\in M_n$, we write $A^\ast$ to denote the conjugate transpose $\bar{A}^\top$ of $A$. A matrix $X\in M_n$ is said to be normal if $XX^\ast=X^\ast X$, a matrix $H\in M_n$ is said to be hermitian if $H^\ast=H$ and a 
matrix $U\in M_n$ such that $U^\ast U=UU^\ast=\mathbf{1}_n$ is called unitary. We will write $\mathbf{i}$ to denote the number $\sqrt{-1}$. 

Let $(X,d)$ be a metric space. We say that $\tilde{X}_\delta\subset X$ is a $\delta$-dense subset of $X$ if for all $x\in X$ there exists $\tilde{x}\in \tilde{X}_\delta$ such that $d(x,\tilde{x})\leq\delta$. 
Given two compact subsets $X,Y$ of the complex plane, we will write $d_H(X,Y)$ to denote the Hausdorff distance between $X$ and $Y$ that is defined by the expression 
$d_H(X,Y):=\max\{sup_{x\in X} \inf_{y\in Y}|x-y|,sup_{y\in Y} \inf_{x\in X}|y-x|\}$. Given $N\in \mathbb{Z}^+$, we will write $\mathbb{X}_N$ to denote the finite set $\{x_k=-1+2k/(N-1)|0\leq k\leq N-1\}\subseteq [-1,1]$.

Given $d,m\in \mathbb{Z}^+$ with $d\leq m$, a compact manifold $Z^d\subseteq \mathbb{R}^m$, and a coordinate chart $\varphi:\mathbb{R}^d\to Z^d\subseteq \mathbb{R}^m,(x_1,\ldots,x_d)\mapsto (\varphi_1(x_1,\ldots,x_d),\ldots,\varphi_m(x_1,\ldots,x_d))$, we will write $\varphi^{-1}_j$ to denote the $j$th component the map $\varphi^{-1}:\mathbb{R}^m\supseteq Z^d\to R^d,(z_1,\ldots,z_m)\mapsto (\varphi_1^{-1}(z_1,\ldots,z_m),\ldots,\varphi_d^{-1}(z_1,\ldots,z_m))$, $1\leq j\leq d$.

From here on $\|\cdot\|$ denotes the operator norm defined for any $A\in M_n$ by $\|A\|:=\sup_{\|x\|_2=1}\|Ax\|_2$, where $\|\cdot\|_2$ denotes the Euclidean norm in $\CC^n$. Let us denote by $\eth$ the metric in $M_n^m$ defined by $\eth:M_n^m\times M_n^m\to \RR^+_0,(\mathbf{S},\mathbf{T})\mapsto \max_j \|S_j-T_j\|$. We will write $\disk[m]$ and $\TT[m]$ to denote the $m$-dimensional closed unit disk and the $m$-dimensional torus respectively. We will write $B(x_0,r)$ to denote the closed ball 
$\{x\in \CC\: | \: |x-x_0|\leq r\}$.

In this document by a matrix contraction we mean a matrix $X$ in $M_n$ such that $\|X\|\leq 1$.  A matrix $P\in M_n$ such that $P^\ast=P=P^2$ is called a projector or projection. Given two projectors $P$ and $Q$, if $PQ=QP=\mathbf{0}_n$ we say that $P$ and $Q$ are orthogonal. By an orthogonal partition of unity in $M_n$, we mean a finite set of pairwise orthogonal projectors $\{P_j\}$ in $M_n$ such that 
$\sum_j P_j=\mathbf{1}_n$. We will omit the explicit reference to $M_n$ when it is clear from the context.

Given any two matrices $X,Y\in M_n$ we will write $[X,Y]$ and $\mathrm{Ad}[X](Y)$ to denote the operations 
$[X,Y]:=XY-YX$ and $\mathrm{Ad}[X](Y):=XYX^*$.

Let $\mathbb{GL}_n$ denote the set of invertible elements in $M_n$. Given a matrix $A\in M_n$, we write $\sigma(X)$ to denote the set $\{\lambda\in \CC \: | \: A-\lambda\mathbf{1}_n\notin \mathbb{GL}_n\}$ of eigenvalues of $A$, the set $\sigma(A)$ is called the spectrum of $A$. Given a compact set $\mathbb{X}\subset \mathbb{C}$ and a subset $S\subseteq M_n$, let us denote by $S(\mathbb{X})$ the set 
$S(\mathbb{X}):=\{X\in S|\sigma(S)\subseteq \mathbb{X}\}$, for instance, the expression $\N{n}(\disk)$ is used to denote the set 
normal contractions in $M_n$.

As a consequence of \cite[T.VI.3.3]{Bhatia_mat_book} and 
\cite[C.VI.3.4]{Bhatia_mat_book} we have that.

\begin{proposition}
\label{spectral_variation_inequality_1}
If $A$ and $B$ are $n\times n$ normal matrices, then 
\[
d_H(\sigma(A),\sigma(B))\leq \|A-B\|.
\]
\end{proposition}

\begin{definition}
Given $\varepsilon\geq 0$ and a matrix $X\in M_n$, we write $\sigma_\varepsilon(X)$ to denote the $\varepsilon$-Pseudospectrum of $X$ which is the set defined by the following relations.
\begin{eqnarray*}
\sigma_\varepsilon(X)&:=&\{\tilde{\lambda}\in \CC \: | \: \tilde{\lambda}\in \sigma(X+E), \mathrm{for \: some} \: E\in M_n \: \mathrm{with } \: \|E\|\leq \varepsilon\}\\
&=&\{\tilde{\lambda}\in \CC \: | \: \|X\mathbf{v}-\tilde{\lambda}\mathbf{v}\|\leq \varepsilon, \mathrm{for \: some} \: \mathbf{v}\in \CC^n \: \mathrm{with } \: \|\mathbf{v}\|=1\}
\end{eqnarray*}
\end{definition}

\begin{definition}[Semialgebraic Matrix Varieties]
\label{matrix_variety}
Given $J\in \ZZ^+$, a system 
of $J$ polynomials $\NS{p}{J}\in \Pi_{\langle N \rangle}=\CC\langle x_1,\ldots,x_N\rangle$ in $N$ noncommutative variables  
and $J$ real numbers $\varepsilon_j\geq 0$, $1\leq j\leq J$, a matrix representation of the noncommutative semialgebraic set $\mathcal{Z}_{n}(\NS{p}{J})$ 
described by 
\begin{equation*}
 \mathcal{Z}_{n}(\NS{p}{J}):=\{\NS{X}{N}\in M_{n} | \|p_j(\NS{X}{N})\|\leq \varepsilon_j, 1\leq j\leq J\},
\end{equation*}
will be called a {\bf $n$-semialgebraic matrix variety} ($n$-SMV), if each $\varepsilon_j=0$, we will refer to the set as a {\bf matrix variety} and we may replace the normed polynomial relations by polynomial relations. 
\end{definition}

\begin{example} As a first example, we will have that the 
matrix set 
\[
\mathbf{Z}_n:=\left\{(X_1,\ldots,X_N)\in M_n^N \left|
\begin{array}{l}
X_jX_k-X_kX_j=\mathbf{0}_n,\\
X_j^*X_j=X_jX_j^*=\mathbf{1}_n,\\
\end{array}
1\leq j,k\leq N
\right.
\right\}
\]
is a matrix variety. If for some $\delta>0$, we set now 
\[
\mathbf{Z}_{n,\delta}:=\left\{(X_1,\ldots,X_N)\in M_n^N \left|
\begin{array}{l}
\|X_jX_k-X_kX_j\|\leq \delta,\\
\|X_j^*X_j-X_jX_j^*\|=0,\\
\|X_j\|\leq 1
\end{array}
1\leq j,k\leq N
\right.
\right\}
\]
the set $\mathbf{Z}_{n,\delta}$ is 
a matrix semialgebraic variety.
\end{example}

\begin{example} Other example of a matrix semialgebraic variety, that has been useful to understand the geometric nature of the 
problems solved in this document, is described by the matrix set $\mathbf{Iso}_{\delta}(x,y)$, 
defined for some given $\delta\geq 0$ and any two normal contractions 
$x$ and $y$ in $M_n$, by the expression
\[
 \mathbf{Iso}_{\delta}(x,y):=
\left\{(z,w)\in\N{n}(\disk)\times\U{n}\left|\: 
\begin{array}{c}
 \|xw-wz\|=0 \\
 \|[z,y]\|=0\\
\|z-y\|\leq \delta
\end{array}
\right.
\right\}.
\]
\end{example}

Given a $m$-tuple $\mathbf{X}\in M_n^m$ we will write either $\mathbf{X}[j]$ or $X_j$ to denote its $j$th component in $M_n$, with $1\leq j\leq m$. In a similar way given a path $\gamma\in C([0,1],M_n^m)$ we will write $\gamma_j$ to denote its $j$th component in $C([0,1],M_n)$, and we will write $\gamma_j(t)$ to denote the $j$th component of $\gamma(t)$ for some $0\leq t\leq 1$, with $1\leq j\leq m$.

Given a map $\Phi:M_n\to M_n$, we will write $\hat{\Phi}$ to denote its {\bf extension} to $M_n^m$ defined by $\hat{\Phi}:M_n^m\to M_n^m,(X_1,\ldots,X_m)\mapsto (\Phi(X_1),\ldots,\Phi(X_m))$.

\begin{definition}[Morphisms of Matrix Varieties]
Given a matrix variety $Z^m\subseteq M_n^m$, a morphism of $Z^m$ is any map $\Psi:Z^m\to Z^m$. Given 
a map $\Psi:M_n\to M_n$, if the extended map $\hat{\Psi}:M_n^m \to M_n^m$ is a morphism of $Z^m$, we say that $\hat{\Psi}$ is a morphism induced by $\Psi$. A morphism $\Phi:Z^m\to Z^m$ is said to be an inner morphism of $Z^m$ if it can be represented in the form $\Phi(X_1,\ldots,X_m)=(\mathrm{Ad}[W_1](X_1),\ldots,\mathrm{Ad}[W_m](X_m))$ for some fixed but arbitrary $W_1,\ldots,W_m\in M_n$ and any $(X_1,\ldots,X_m)\in Z^m$.
\end{definition}

\begin{remark}
\label{inner_morphisms_remark}
In particular, given a matrix varietiy $Z^m\subseteq M_n^m$, and given any inner morphism $\Phi:Z^m\to Z^m$ of the form $\Phi=\widehat{\mathrm{Ad}[W]}$ for some $W\in \U{n}$, it can be seen that if the components $X_j$ of $(X_1,\ldots,X_m)\in Z^m$ are normal matrices, unitaries, commuting matrices, hermitian matrices, projectors, or contractions, then the same will be true for the components $\mathrm{Ad}[W](X_j)$ of $\Phi(X_1,\ldots,X_m)\in Z^m$. Moreover, $\sigma(X_j)=\sigma(\mathrm{Ad}[W](X_j))$ for each $1\leq j\leq m$. 
\end{remark}

\begin{definition}[$\circledast$ operation]
 Given two matrix paths $\alpha,\beta\in C([0,1],M_n^m)$ we write ${\alpha \circledast \beta}$ to denote the concatenation of $\alpha$ and $\beta$, which is 
 the matrix path defined in terms of $\alpha$ and $\beta$ by the expression,
 \[
  {\alpha\circledast \beta}(s):=
  \left\{
  \begin{array}{l}
   \alpha(2s),\:\: 0\leq s\leq \frac{1}{2},\\
   \beta(2s-1),\:\: \frac{1}{2}\leq s\leq 1.
  \end{array}
  \right.  
 \]
\end{definition}

From here on, for any two $m$-tuples of matrices $\mathbf{X},\mathbf{Y}\in M_n^m$, we will write $\mathbf{X}\leadsto \mathbf{Y}$ to indicate that there is a matrix path $\gamma\in C([0,1],M_n^m)$ such that $\gamma(0)=\mathbf{X}$ and $\gamma(1)=\mathbf{Y}$.

Given $\varepsilon,\delta>0$, and any two $m$-tuples of pairwise commuting normal matrix contractions $\mathbf{X}:=(X_1,\ldots,X_m)$ and $\tilde{\mathbf{X}}:=(\tilde{X}_1,\ldots,\tilde{X}_m)$ in some (semialgebraic) matrix variety $Z^m\subseteq M_n^m$ such that 
$\eth(\mathbf{X},\tilde{\mathbf{X}})\leq \delta$, we will write $\mathbf{X}\leadsto_{\varepsilon} \tilde{\mathbf{X}}$ {\bf \em relative to} $B_\eth(\mathbf{X},\varepsilon)\cap Z^m$, to indicate that there is a piecewise differentiable path $\gamma\in C([0,1],B_\eth(\mathbf{X},\varepsilon)\cap Z^m)$ such that $\gamma(0)=\mathbf{X}$ and $\gamma(1)=\mathbf{\tilde{X}}$.

Given a matrix $A\in M_n$, we will write $\mathscr{D}(A)$ to denote the diagonal matrix defined by the following operation.
\begin{eqnarray}
\mathscr{D}(A)&:=&\mathrm{diag}[a_{11},a_{22},\ldots,a_{nn}]\\
&=&
\left(
\begin{array}{cccc}
a_{11} & 0 & \cdots & 0\\
0 & a_{22} & \cdots & \vdots\\
\vdots & \ddots & \ddots & 0\\
0 & \cdots & 0 & a_{nn}
\end{array}
\right)
\end{eqnarray}

It can be seen that $\mathscr{D}(\mathscr{D}(A))=\mathscr{D}(A)$ for any $A\in M_n$, the map $\mathscr{D}$ will be called the full pinching. 

\begin{remark}
By pinching inequalities (in the sense of \cite{Bhatia_matrix_inequalities}) we will have that $\|\mathscr{D}(A)-\mathscr{D}(B)\|=\|\mathscr{D}(A-B)\|\leq \|A-B\|$ for any two matrices $A,B\in M_n$.
\end{remark}

It is often convenient to have $N$-tuples (or $2N$-tuples) of matrices with real spectra. For this purpose we use the following construction. If $X=(\NS{X}{N})$ is a $N$-tuple of $n$ by $n$ matrices then we can always decompose $X_j$ in the form 
$X_j=X_{1j}+iX_{2j}$ where the $X_{kj}$ all have real spectra. We write 
$\pi(X):=(X_{11},\ldots,X_{1N},X_{21},\ldots,\\
X_{2N})$ and call $\pi(X)$ a {\bf partition} of $X$. If the $X_{kj}$ all commute we say that 
$\pi(X)$ is a commuting partition, and if the $X_{kj}$ are simultaneously triangularizable $\pi(X)$ is a triangularizable partition. If 
the $X_{kj}$ are all semisimple (diagonalizable) then $\pi(X)$ is called a semisimple partition. From here on, for any $m$-tuple of pairwise commuting normal matrices $X=(X_1,\ldots,X_m)$, we will assume that the components of $\pi(X)=(X_{11},\ldots,X_{1m},X_{21},\ldots,X_{2m})$ are determined by the relations
\begin{equation}
\left\{
\begin{array}{l}
X_{1j}=(X_j+X_j^\ast)/2,\\
X_{2j}=(X_j-X_j^\ast)/{(2\mathbf{i})},
\end{array}
\right.
\label{hermitian_partition_relations}
\end{equation}
for each $1\leq j\leq m$. In addition, given any operator $X$ in a unital separable $C^\ast$-algebra $A$, we will write $Re(X)$ and $Im(X)$ to denote the self-adjoint operators in $A$ defined by the equations $Re(X)=(X+X^\ast)/2$ and $Im(X)=(X-X^\ast)/(2\mathbf{i})$.

Given a $2m$-tuple $\mathbf{X}=(X_{11},\ldots,X_{1m},X_{21},\ldots,X_{2m})$ in 
$M_n^{2m}$, the $m$-tuple obtained by the operation 
$\upsilon(\mathbf{X}):=(X_{11}+\mathbf{i}X_{21},\ldots,X_{1m}+\mathbf{i}X_{2m})\in M_n^m$ will be called 
{\bf juncture} of $\mathbf{X}$.

We say that $N$ normal matrices $\NS{X}{N}\in M_n$ are {\em simultaneously diagonalizable} if there is a unitary matrix 
$Q\in M_n$ such that 
$Q^* X_jQ$ is diagonal for each $j=1,\ldots,N$. In this case, for $1\leq k\leq n$, let 
$\Lambda^{(k)}(X_j):=(Q^*X_jQ)_{kk}$ the $(k,k)$ element of $Q^*X_jQ$, and set 
$\Lambda^{(k)}(\NS{X}{N}):=(\Lambda^{(k)}(X_1),\ldots,\Lambda^{(k)}(X_N))$ in $\CC^N$. The set
\[
 \Lambda(\NS{X}{N}):=\{\Lambda^{(k)}(\NS{X}{N})\}_{1\leq k\leq N}
\]
is called the joint spectrum of $\NS{X}{N}$. Given a set $S\subseteq M_n$ of $m$-tuples of pairwise commuting normal matrices, we will write $\Lambda(S)$ to denote the set $\{\Lambda(X) \: |\: X\in S\}$, the set $\Lambda(S)$ will be called the joint spectra of $S$. We will write $\Lambda(X_j)$ to denote the diagonal matrix representation of the $j$-component of $\Lambda(\NS{X}{N})$, in other words 
we will have that
\[
 \Lambda(X_j)=\diag{\Lambda^{(1)}(X_j),\ldots,\Lambda^{(n)}(X_j)}.
\]

Given a $m$-tuple $\mathbf{X}=(X_1,\ldots,X_m)\in M_n^m$ of commuting normal matrices, any orthogonal projection $P\in M_n$ such that 
$X_jP=PX_j=\Lambda^{(r)}(X_j)P$ for each $1\leq j\leq m$ and some $1\leq r\leq n$, will be called a joint spectral projector of $\mathbf{X}$.

\begin{definition}[$C^\ast$-algebras generated by $m$-tuples] 
Given a $m$-tuple of normal matrix contractions $\mathbf{X}=(X_1,\ldots,X_m)\in M_n^m$ we write $C^\ast(\mathbf{X})$ to denote 
the $C^\ast$-algebra $C^\ast(X_1,\ldots,X_m)\subseteq M_n$. 
\end{definition}

The following result (\cite[Lemma 1]{H_Lin_Approx_Fin_Spec}) was proved in \cite{H_Lin_Approx_Fin_Spec}.

\begin{lemma}
\label{cont_funct_calc_tool}
Let $X$ be a compact sucset of the plane and let $f\in C(X)$. For any $\varepsilon>0$, there is $\delta>0$, for any $C^\ast$-algebra $A$ and normal elements $x,y\in A$ with $\sigma(x),\sigma(y)\in X$, if $\|x-y\|<\delta$, then $\|f(x)-f(y)\|<\varepsilon$.
\end{lemma}

The following result (\cite[Lemma 4.1]{Vides_homotopies}) was proved in \cite{Vides_homotopies}.

\begin{lemma}\label{Joint_spectral_variation_inequality_2}
 Given $\varepsilon>0$ there is $\delta=\frac{1}{K_m}\varepsilon> 0$ such that, for any two $N$-tuples of pairwise commuting normal 
 matrices $\mathbf{X}=(\NS{X}{N})$ and $\mathbf{Y}:=(\NS{Y}{N})$ 
 such that $\eth(\mathbf{X},\mathbf{Y})\leq \delta$, there is a unitary matrix $W$ such that the $\ast$-homomorphism $\Psi=\mathrm{Ad}[W]$ satisfies the condition $\Psi:C^\ast(\mathbf{X})\to C^\ast(\mathbf{Y})'$ together with the constraint 
 $\max\{\eth(\hat{\Psi}(\mathbf{X}),\mathbf{Y}),\eth(\hat{\Psi}(\mathbf{X}),\mathbf{X})\}\leq \varepsilon$.
\end{lemma}

\begin{remark}
The constant $K_m$ in the statement of L.\ref{Joint_spectral_variation_inequality_2} depends only on $m$.
\end{remark}

\section{Spectral Approximation and Local Path Connectivity of Sets of Commuting Normal Matrices}
\label{main_results}

\subsubsection{Matrix Representations of $C(\mathbb{X}_N)$} Let us consider the universal commutative $C^\ast$-algebra $C(\mathbb{X}_N)$ for some fixed but arbitrary $N\in \ZZ^+$. We will have that $C(\mathbb{X}_N)$ has a universal picture that can be expressed in terms of generators and relations as follows.
\begin{equation}
C(\mathbb{X}_N):=
 C^*_1\left\langle p_1,\ldots,p_N \left| 
                     \begin{array}{l}
                        p_j p_k=p_k p_j\\
                       p_j^2-p_j=p_j-p_j^\ast=0\\
                       p_1^2+\ldots+p_N^2-1=0
                      \end{array}, 1\leq j,k\leq N
\right.\right\rangle,
\label{spheres_of_projections}
\end{equation}
From here on, given a universal $C^\ast$-algebra $A$ and a $C^\ast$-subalgebra $B\subseteq M_n$, we will write $\mathrm{C^\ast Rep}(A,B)$ to denote the set of $\ast$-representations of $A$ in $B$. 

\subsubsection{Eventually uniform matrix representations} \label{hard_approximation}
Let $N\in \ZZ^+$ be fix but arbitrary. Given $n\in \ZZ^+$, let us consider the case $n<N$ first. Without loss of generality we can consider the $\ast$-homomorphism induced by the embedding $\imath_{n,N}:M_n\hookrightarrow M_N$ determined by the map
\begin{eqnarray}
\imath_{r,s}&:&M_r\to M_{s}\nonumber \\
&:&X\mapsto \left(
\begin{array}{cc}
X & \mathbf{0}_{n,N-n}\\
\mathbf{0}_{N-n,n} & \mathbf{0}_{N-n}
\end{array}
\right)
\label{canonical_matrix_embeddding}
\end{eqnarray}
for any two positive integers $r\leq s$. We can use \eqref{canonical_matrix_embeddding} to compute for any orthogonal partition of unity $\mathbf{P}_n=\{P_1,\ldots,P_n\}\subseteq M_n$, a set of orthogonal projections  $\mathbf{p}_n=\{p_1,\ldots,p_n\}=\{\imath_{n,N}(p_1),\ldots,\imath_{n,N}(p_n)\}=\{P_1\oplus \mathbf{0}_{N-n},\ldots,P_n\oplus \mathbf{0}_{N-n}\}\subseteq M_N$. We will have that the map 
$\kappa_{N,n}:M_N\to M_n:X\mapsto \mathrm{Ad}[K](X)$, with $K=(\I_n \:\:\: \mathbf{0}_{n,N-n})$, induces a $\ast$-homomorphism $\rho^>_{N,n}:C^\ast(\mathbf{p}_n)\to C^\ast(\mathbf{P}_n)$ such that $\rho^>_{N,n}(p_j)=P_j$ for each $1\leq j\leq n$, by universality of 
\eqref{spheres_of_projections} the map $\rho^>_{N,n}$ induces a 
$\ast$-homomorphism $\rho_n:C(\mathbb{X}_N)\to C^\ast(\mathbf{P}_n)\subseteq M_n$, for each $n<N$. If $n\geq N$, by universality of \eqref{spheres_of_projections}, we will have that for any orthogonal partition of unity $\mathbf{P}_n=\{P_1,\ldots,P_N\}\subseteq M_n$, there is a $\ast$-homomorphism $\rho^{\leq}_{N,n}:C(\mathbb{X}_N)\to C^\ast(\mathbf{P}_n)\subseteq M_n$ such that $\rho^{\leq}_{N,n}(p_j)=P_j$ for each $1\leq j\leq N$. Let us set
\begin{equation}
\rho_{N,n}:=
\left\{
\begin{array}{l}
\rho^>_{N,n},\:\: n<N\\
\rho^\leq_{N,n},\:\: n\geq N
\end{array}
\right..
\label{eventual_connecting_morphism}
\end{equation}
It can be seen that the map $\rho_{N,n}$ is a $\ast$-homomorphism for each $n,N\in \mathbb{Z}^+$. We will call 
each map $\rho_{N,n}$ of the sequence $\{\rho_{N,n}\}_{n\in \ZZ^+}$ an {\bf eventually uniform} matrix representation ({\bf EUMR}) of $C(\mathbb{X}_N)$.

In order to provide some perspective on how our connectivity and approximation results do not depend on the matrix size $n$, we will describe what we call the {\bf approximate eventual uniformity} of matrix representaions of $C(X)$, where $X$ is a compact Hausdorff topological space. 

\begin{definition}
Given a compact Hausdorff space $X$, we say that the commutative $C^\ast$-algebra $C(X)$ is approximately eventually uniform ({\bf AEU}) if for any finite set of normal contractions $\{z_1,\ldots,z_m\}\subset C(X)$, there is $\delta>0$ and an integer $N_{\delta,m}>0$ such that 
for any integer $n>1$ and any $\rho_n\in \mathrm{C^\ast Rep}(C(X),M_n)$, there is a {\ EUMR} $\rho_{N_{\delta,m},n} \in \mathrm{C^\ast Rep}(C(\mathbb{X}_{N_{\delta,m}}),\rho_n(C(X))')$ together with a $\ast$-homomorphism 
$\rho_{\delta,n}:C(X)\to C^*(\rho_{N_\delta,n}(\mathrm{id}_{\mathbb{X}_{N_\delta}}))$ such that $\|\rho_n(z_j)-\rho_{\delta,n}(z_j)\|\leq \delta$ for each $1\leq j\leq m$.
\end{definition}

By considering natural embeddings of compact Hausdorff spaces into $[-1,1]^m$ for some $m\in \mathbb{Z}^+$, we have that one of the main problems that we need to solve in this document, is an operator approximation problem that can be stated via the following diagram.
\begin{equation}
\xymatrix{
  C(\mathbb{X}_{N_\delta}) \ar@{-->}[d] & C(\mathbb{I}^{m}) \ar@{-->}[ld]_{\approx_\delta} \ar[d] \ar[rd] \ar[d] & C(\mathbb{X}_{n}) \ar[d]\\
  C^\ast(\hat{\mathbb{X}}_{N_\delta}) \ar[rd]_{\imath_1} & C^\ast(H_1,\ldots,H_{m}) \ar[d]_{\imath_2} & C^\ast(\hat{\mathbb{X}}_{n}) \ar@{^{(}->}[d]\\
  & C^\ast(\hat{\mathbb{X}}_{N_\delta},H_1,\ldots,H_{m}) \ar[ru] & M_n
}
\label{first_approximation_problem}
\end{equation}
To {\bf solve} \eqref{first_approximation_problem}, given a $\ast$-representation $\pi:C(\mathbb{I}^{m})\to C^\ast(H_1,\ldots,H_{m})$ such that $\pi(h_j)=H_j$ for each $1\leq j\leq m$, one needs to find $\pi_n\in \mathrm{C^\ast Rep}(C(\mathbb{X}_{N_\delta}),\pi(C(\mathbb{I}^m))')$ and a $\ast$-representation $\pi_{\delta}:C(\mathbb{I}^{m})\to C^\ast(\hat{\mathbb{X}}_{N_\delta})$ such that: 
\begin{equation}
\left\{
\begin{array}{l}
(\imath_2\circ\pi)(h_j)(\imath_1\circ\pi_\delta)(h_k)=(\imath_1\circ\pi_\delta)(h_k)(\imath_2\circ\pi)(h_j)\\
\|(\imath_2\circ\pi)(h_j)-(\imath_1\circ\pi_\delta)(h_j)\|\leq \delta
\end{array}, 1\leq j,k\leq m.
\right.
\label{first_approximation_problem_solution}
\end{equation}
Here $\hat{\mathbb{X}}_{N_\delta},\hat{\mathbb{X}}_{n}\in C^*(H_1,\ldots,H_m)'$ are defined by the equations $\hat{\mathbb{X}}_{N_\delta}=\pi_n(\mathrm{id}_{\mathbb{X}_{N_\delta}})$ and $\hat{\mathbb{X}}_{n}=\rho_n(\mathrm{id}_{\mathbb{X}_{n}})$ for some $\ast$-homomorphism $\rho_{n}\in \mathrm{C^*Rep}(C(\mathbb{X}_{n}),C^*(H_1,\ldots,H_m)')$.

Let us define two semialgebraic matrix varieties that will play a key role in this study.

\begin{definition} The semialgebraic matrix variety $\mathbb{I}^m(n)$ defined by
\[
\mathbb{I}^m(n):=
\left\{
(X_1,\ldots,X_m)\in M_n^m \left|
\begin{array}{l}
[X_j,X_k]=\mathbf{0}_n\\
X_j-X_j^\ast=\mathbf{0}_n\\
\|X_j\|\leq 1
\end{array}
, 1\leq j,k \leq m\right.
\right\}
\]
will be called the {\em \bf matrix $m$-cube}.
\end{definition}

\begin{definition} The semialgebraic matrix variety $\mathbb{D}^m(n)$ defined by
\[
\mathbb{D}^m(n):=
\left\{
(X_1,\ldots,X_m)\in M_n^m \left|
\begin{array}{l}
[X_j,X_k]=\mathbf{0}_n\\
X_jX_j^\ast-X_j^\ast X_j=\mathbf{0}_n\\
\|X_j\|\leq 1
\end{array}
, 1\leq j,k \leq m\right.
\right\}
\]
will be called the {\em \bf matrix $m$-disk}.
\end{definition}

Let us now extend the concept of semialgebraic matrix variety to what we will call an {\em induced compact differential matrix manifold} or {\em compact differential matrix manifold}.

\begin{definition} 
\label{induced_matrix_manifold}
Given a compact differentiable $d$-manifold $Z^d\subseteq \RR^m$ with $m\geq d$, the matrix variety $\mathbb{V}^m[Z^d](n)$ defined by
\[
\mathbb{V}^m[Z^d](n):=
\left\{
(X_1,\ldots,X_m)\in M_n^m \left|
\begin{array}{l}
[X_j,X_k]=\mathbf{0}_n\\
X_j-X_j^\ast=\mathbf{0}_n\\
\|X_j\|\leq 1\\
\Lambda(X_1,\ldots,X_m)\subseteq Z^d 
\end{array}
, 1\leq j,k \leq m\right.
\right\}
\]
will be called the {\em \bf differentiable matrix manifold} induced by $Z^d$.
\end{definition}

\subsubsection{Borel functional calculus and matrix approximation} By applying Borel functional calculus on commutative von Neumann algebras we have that given any real number $0<\delta \leq 1$, any hermitian matrix $X\in M_n$ such that 
$\sigma(X)\subset [-1,1]$, and any partition $-1+\delta/2=x_1<x_2<\cdots<x_m<x_{m+1}=1+\delta/2$ of $[-1-\delta/2,1+\delta/2]$, with $x_{k+1}-x_{k}=\delta$ for each $1\leq k\leq m+1$, the matrix $X$ has a representation $X=\sum_{k=1}^{m+1}\chi_{[x_k,x_{k+1})}(X)X$, where $\chi_{[a,b)}(X)$ denotes the characteristic function of the interval $[a,b)$, for some real numbers $a,b$ such that $a<b$. We also have that for each $1\leq k\leq m$, the matrix $P_k=\chi_{[x_k,x_{k+1})}(X)$ is an orthogonal projection that commutes with $X$ such that $\sigma(P_kX)\subseteq \sigma(X)\cap [x_k,x_{k+1})$. Moreover, the set $\{P_1.\ldots,P_m\}$ is an orthogonal partition of unity.

\begin{definition}[Pseudospectral matrix approximant] Given $\delta>0$ and any matrix $X\in M_n$, we say that the matrix $\tilde{X}\in M_n$ is a $\delta$-Pseudospectral matrix approximant ({\bf $\delta$-PMA}) of $X$ if $\|X-\tilde{X}\|\leq \delta$ and 
$\sigma(\tilde{X})$ is $\delta$-dense in $\sigma(X)$. If in addition each $\tilde{X}X=X\tilde{X}$, we say that $\tilde{X}$ is a commuting $\delta$-PMA ({\bf $\delta$-CPMA}).
\end{definition}

\begin{lemma}
\label{Existence_of_1D_PMA}
Given $\delta>0$ and any hermitian matrix $X\in M_n$ such that $\|X\|\leq 1$, there is a hermitian $\delta$-CPMA $\tilde{X}_\delta$ of $X$. Moreover, there are $N_\delta$
numbers $x_1,\ldots,x_{N_\delta}$
in $[-1,1]$ and an orthogonal partition of unity $\{P_1,\ldots,P_{N_\delta}\}$ such that $\tilde{X}_\delta=\sum_j x_j P_j$.
\end{lemma}
\begin{proof}
Let us suppose that $n\geq |\sigma(X)|\geq 2$, as the proof is clear for scalar multiples of $\mathbf{1}_n$. Since $\sigma(X)\subseteq [-1,1]$, we can assume for simplicity that $M_\delta:=1+(\delta)^{-1} \in \mathbb{Z}^+$. Then the finite set 
$\hat{R}_\delta(X):=\{x_k:=-1+2(k-1)\delta\: | \: 1\leq k\leq M_\delta\}$ is $\delta$-dense in $[-1,1]$ with $|\hat{R}_\delta(X)|=M_\delta$. Let us set $S_\delta(X):=\{\check{x}_k:=-1+(2k-3)\delta\: | \: 1\leq k \leq M_{\delta}+1\}$ and 
$P_k:=\chi_{(\check{x}_k,\check{x}_{k+1}]}(X)$ for each $1\leq k\leq M_\delta$. Then there is a set $R_\delta(X):=\{x_j\}\subseteq \hat{R}_\delta(X)$ that is $\delta$-dense in $\sigma(X)$ with $x_j\in (\check{x}_{k(j)},\check{x}_{k(j)+1}]$ for each $x_j\in R_\delta(X)$ and some $\check{x}_{k(j)},\check{x}_{k(j)+1}\in S_\delta(X)$, and the set $P_\delta(X):=\{P_1,\ldots,P_{N_\delta}\}\subseteq M_n\backslash \{\mathbf{0}_n\}$ of $N_\delta=|R_\delta(X)|$ projections is an orthogonal partition of unity such that $[P_j,X]=\mathbf{0}_n$ for each $j$. We will have that for each $1\leq k\leq N_\delta\leq M_\delta$ there is  $x_{j(k)}\in R_\delta(X)$ such that $\|XP_k-x_{j(k)}P_k\|\leq \delta$. By the previous facts and Proposition \ref{spectral_variation_inequality_1} we will have that if we set $\tilde{X}_\delta:=\sum_k x_{j(k)}P_k$ with $x_{j(k)}\in R_\delta(X)$ for each $k$, then $[X,\tilde{X}_\delta]=\mathbf{0}_n$ and $d_H(\sigma(X),\sigma(\tilde{X}_\delta))\leq \|X-\tilde{X}_\delta\|\leq\max_k \|XP_k-x_{j(k)}P_k\|\leq \delta$. If necessary we can renumber the elements of $R_\delta(X)$ according to the elements of $P_\delta(X)$. This completes the proof.
\end{proof}

\begin{figure}[!htb]
\centering
 \includegraphics[scale=0.6]{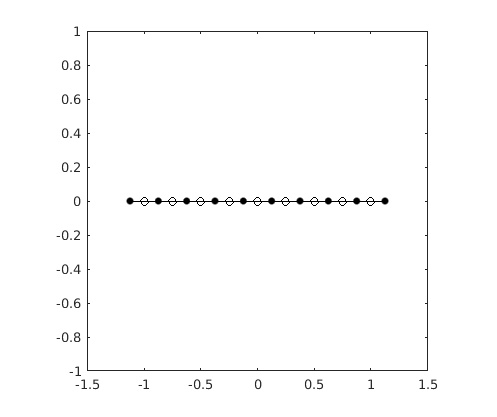}
 \caption{Graphical illustration of a representation grid (consisting of $\circ$ points) for $[0,1]$ together with its support grid (consisting of $\bullet$ points).}
 \label{dual_grids_1d}
 \end{figure}

\begin{figure}[!htb]
\centering
 \includegraphics[scale=0.6]{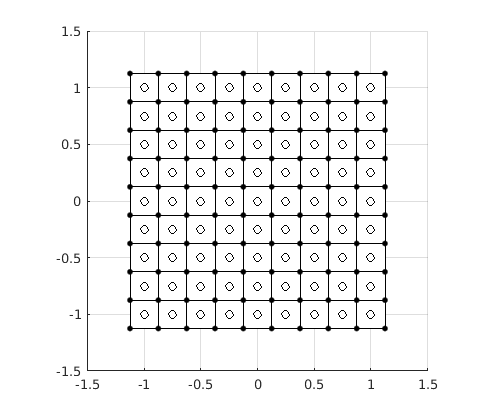}
 \caption{Graphical illustration of a representation grid (consisting of $\circ$ points) for $[0,1]^2$ together with its support grid (consisting of $\bullet$ points).}
 \label{dual_grids_2d}
 \end{figure}
 
 \begin{definition}[CPMA for $m$-tuples] Given $\delta>0$ and any $m$-tuple of pairwise commuting matrices $\mathbf{X}=(X_1,\ldots,X_m)\in M_n^m$, we say that the matrix $\tilde{\mathbf{X}}=(\tilde{X}_1,\ldots,\tilde{X}_m)\in M_n^m$ is a commuting $\delta$-Pseudospectral matrix approximant ({\bf $\delta$-CPMA}) of $\mathbf{X}$ if $\|X_j-\tilde{X}_j\|\leq \delta$, $\sigma(\tilde{X}_j)$ is $\delta$-dense in $\sigma(X_j)$ and  $\tilde{X}_kX_j=X_j\tilde{X}_k$ for each $1\leq j,k\leq m$. If in addition each component $\tilde{X}_j$ of $\tilde{\mathbf{X}}$ is normal, hermitian or unitary, we say that $\tilde{\mathbf{X}}$ is a hermitian, normal or unitary $\delta$-CPMA respectively.
\end{definition}

\begin{lemma}
\label{Existence_of_mD_PMA}
Given $\delta>0$, an integer $m\geq 1$ and any $m$-tuple of pairwise commuting hermitian matrix contractions $\mathbf{X}=(X_1,\ldots,X_m)\in M_n^m$, there is a hermitian $\delta$-CPMA  $\tilde{\mathbf{X}}_\delta=(\tilde{X}_{\delta,1},\ldots,\tilde{X}_{\delta,m})\in M_n^m$ of $\mathbf{X}$. Moreover, for each $1\leq k\leq m$, there are $N_\delta$ (not necessarily distinct) numbers $x_{k,1},\ldots,x_{k,N_\delta}$ counted with multiplicity in $[-1,1]$ and an orthogonal partition of unity $\{P_1,\ldots,P_{N_\delta}\}$ such that $\tilde{X}_{\delta,k}=\sum_j x_{k,j} P_j$.
\end{lemma}
\begin{proof}
Since the joint spectrum $\sigma(\mathbf{X})$ of $\mathbf{X}$ is a subset of $[-1,1]^m$ we can apply a similar procedure to the one implemented in the proof of Lemma \ref{Existence_of_1D_PMA} to find for each $X_j$ a representation grid ${R}_{\delta}(X_j)$ together with a support grid $S_{\delta}(X_j)$, an associated 
$\delta$-projective decomposition $P_\delta(X_j)$ and an integer $N_{\delta,j}:=|P_\delta(X_j)|=|{R}_{\delta}(X_j)|$. We will have that 
the set $P_\delta(\mathbf{X}):=\{P_1,\ldots,P_{N_\delta}\}=\{P_{1,j_1}P_{2,j_2}\cdots P_{m,j_m}\:|\: P_{k,j_k}\in P_\delta(X_k), \: 1\leq j_k\leq N_{\delta,k}, \: 
1\leq k \leq m\}$ is an orthogonal partition of unity such that $PQ=QP$ for each $P\in P_\delta(\mathbf{X})$ and 
each $Q\in P_\delta(X_k)$, for any $1\leq k\leq m$. By setting $N_{\delta}:=|P_\delta(\mathbf{X})|$, we will have that $N_\delta\geq N_{\delta,k}$ for each $k$. We will also have that $P_\delta(X_k)\subseteq \mathrm{span}\: P_\delta(\mathbf{X})$ for each $k$. For each $X_k$ and each $P_j\in P_\delta({\mathbf{X}})$ there is $x_{k,j}\in R_\delta(X_j)\cup \{0\}$ such that $\|X_kP_j-x_{k,j}P_j\|\leq \delta$. Moreover,  
$[X_j,P_k]=\mathbf{0}_n$ for each $1\leq j\leq m$ and each $1\leq k \leq N_\delta$. By a similar argument to the one implemented in the proof of Lemma \ref{Existence_of_1D_PMA} it can be seen that the matrix $\tilde{X}_{\delta,k}:=\sum_{j}x_{k,j}P_j$ is a 
commuting hermitian $\delta$-PMA of $X_j$ for each $j$. This completes the proof.
\end{proof}

\begin{remark}
Given a matrix $m$-tuple $\mathbf{X}\in M_n^m$, we will refer to the finite sets $R_\delta(\mathbf{X})=\bigcup_{j} R_\delta(X_j)\cup \{0\}$  and $S_\delta(\mathbf{X})=\bigcup_{j} S_\delta(X_j)$ described in the proof of Lemmas \ref{Existence_of_1D_PMA} and \ref{Existence_of_mD_PMA} as {\em \bf representation} and {\em \bf support} grids respectively. The finite set $P_\delta(\mathbf{X})$ described in the previous lemma will be called a {\bf $\delta$-projective} decomposition of $\mathbf{X}$. Two graphical examples of representation and support grids are presented in Fogires \ref{dual_grids_1d} and \ref{dual_grids_2d}.
\end{remark}

We can now obtain the following approximation result for $C([-1,1]^m)$.

\begin{lemma}
\label{Cubes_are_AEU}
Given $m\in\mathbb{Z}^+$, we have that for any $\delta>0$ the local approximation problem in $C([-1,1]^m)$ described by the diagram \eqref{first_approximation_problem} is solvable.
\end{lemma}
\begin{proof}
Let $m\in \mathbb{Z}^+$. For any $n\in\mathbb{Z}^+$ and any $\ast$-homomorphism $\rho_n$ in the set $\mathrm{C^\ast Rep}(C([-1,1]^m),M_n)$, we can apply Lemma \ref{Existence_of_mD_PMA} to the $m$-tuple $\mathbf{H}=(\pi_n(h_j))\in \mathbb{I}^m(n)$ to obtain a $\delta$-{\bf CPMA} $\tilde{\mathbf{H}}=(\tilde{H}_j)\in \mathbb{I}^m(n)$ such that $C^\ast(\tilde{\mathbf{H}})\subseteq C^\ast(P_\delta(\mathbf{H}))$. Let us set $N_\delta=|P_\delta(\mathbf{H})|$, by universality of $C(\mathbb{X}_{N_\delta})$ we will have that there is a $\ast$-representation $\rho_n:C(\mathbb{X}_{N_\delta})\to C^\ast(P_\delta(\mathbf{H})) \supseteq C^\ast(\tilde{\mathbf{H}})$, and by universality of $C([-1,1]^m)$ we will have that there is a $\ast$-representation $\pi_\delta:C([-1,1]^m)\to C^\ast(\tilde{\mathbf{H}})\subseteq C^\ast(P_\delta(\mathbf{H}))$. It can be seen that $\pi_\delta$ satisfies the conditions \eqref{first_approximation_problem_solution}. This completes the proof.
\end{proof}

\begin{theorem}
\label{Every_CX_is_AEU}
Given any compact Hasudorff space $X$, we have that the universal $C^\ast$-algebra $C(X)$ is {\bf AEU}.
\end{theorem}
\begin{proof}
Given any $m$ contractions $z_{1},\ldots,z_m\in C(X)$, let us consider the $2m$ selfajoint contractions 
$Re(z_1),Im(z_1),\ldots,Re(z_m),Im(z_m)$. By universality of the $C^\ast$-algebra $C([-1,1]^{2m})$ we will have that 
any $\ast$-representation of the form 
\begin{equation*}
\hat{\pi}_n:C^\ast(Re(z_1),Im(z_1),\ldots,Re(z_m),Im(z_m))\to M_n,
\end{equation*}
induces a $\ast$-representation of the form 
\begin{equation}
\pi_n:C([-1,1]^{2m})\to C^\ast(Re(z_1),Im(z_1),\ldots,Re(z_m),Im(z_m))\to M_n,
\label{matrix_model_1}
\end{equation}
determined by the equations $\pi_n(h_{2j-1})=\hat{\pi}_n(Re(z_j))$ and $\pi_n(h_{2j})=\hat{\pi}_n(Im(z_j))$, $1\leq j\leq m$. We can apply Lemma \ref{Cubes_are_AEU} to $\pi_n$ replacing $\delta$ by $\delta/2$ if necessary. This completes the proof.
\end{proof}

\begin{lemma}
\label{Existence_of_mD_PMA_manifolds}
Given $d,m\in\mathbb{Z}^+$ with $d\leq m$, a compact differentiable $d$-manifold $Z^d\subseteq [-1,1]^m\subseteq \RR^m$, a number $\delta>0$, and any $m$-tuple of pairwise commuting hermitian matrix contractions $\mathbf{Z}=(Z_1,\ldots,Z_m)\in \mathbb{V}^m[Z^d](n)$, 
there is a $\delta$-CPMA $\tilde{\mathbf{Z}}_\delta=(\tilde{Z}_{\delta,1},\ldots,\tilde{Z}_{\delta,m})\in \mathbb{V}^m[Z^d](n)$ of $\mathbf{Z}$. Moreover, there are $\mathbf{H}=(H_1,\ldots,H_d)$ and  $\tilde{\mathbf{H}}_{\delta}=(H_{\delta,1},\ldots,\tilde{H}_{\delta,d})$ in $\mathbb{I}^m(n)$ such that for each $1\leq k\leq m$ and each $1\leq k'\leq d$, there are $N_\delta$ (not necessarily distinct) numbers $x_{k',1},\ldots,x_{k',N_\delta}$ counted with multiplicity in $\RR$, $N_\delta$ (not necessarily distinct) coordinate charts $\varphi_{k,1},\ldots,\varphi_{k,N_\delta}$ of $Z^d$ counted with multiplicity, and an orthogonal partition of unity $\{P_1,\ldots,P_{N_\delta}\}$ such that 
\begin{equation}
\left\{
\begin{array}{l}
Z_k=\sum_j \varphi_{k,j}(P_j{H}_{1},\ldots,P_j{H}_{d})P_j\\
\tilde{Z}_{\delta,k}=\sum_j\varphi_{k,j}(P_j\tilde{H}_{\delta,1},\ldots,P_j\tilde{H}_{\delta,d})P_j\\
\varphi_{k,j}(P_j\tilde{H}_{\delta,1},\ldots,P_j\tilde{H}_{\delta,d})=\varphi_{k,j}(x_{1,j},\ldots,x_{d,j}) P_j,
\end{array}
\right.
\label{additional_matrix_manifold_conditions}
\end{equation}
for each $1\leq j\leq N_\delta$ and each $\:\: 1\leq k\leq m$.
\end{lemma}
\begin{proof}
Let $Z^d$ be a compact differentiable $d$-manifold given as in the statement of Lemma \ref{Existence_of_mD_PMA_manifolds}. Let us consider a $m$-tuple of pairwise commuting hermitian matrix contractions $\mathbf{Z}=(Z_1,\ldots,Z_m)\in \mathbb{V}^m[Z^d](n)$. Since $Z^d$ is a compact differentiable $d$-manifold, for any $\delta>0$, there is a (finite) family of differentiable charts 
$\mathcal{C}(Z^d):=\{\varphi_{k,l}\}$ of $Z^d$ together with a number $\nu_\delta>0$ such that 
\begin{equation}
(\varphi_{1,l}(x),\ldots,\varphi_{m,l}(x))\in B^\circ((\varphi_{1,l}(x),\ldots,\varphi_{m,l}(x));\delta)\cap Z^d,
\label{generic_matrix_neighborhoods}
\end{equation}
for each $x\in B^\circ(x,\nu_\delta)\cap dom(\varphi_{k,l})\subset \mathbb{R}^d$ and for each $k,l$, where $B^\circ(x,r)$ denotes the open ball of radius $r$ in the Euclidean metric in $\RR^m$. For each $1\leq k\leq m$, let us consider the spectral resolutions 
\begin{equation}
Z_k=\sum_{i=1}^n z_{k,i}Q_{i},
\label{manifold_spectral_resolution}
\end{equation}
where each $Q_i\in M_n$ is a joint rank one spectral projector of $\mathbf{Z}$, by Definition \ref{induced_matrix_manifold} and by \eqref{generic_matrix_neighborhoods} we will have that there is a hermitian matrix contraction 
\begin{equation}
H_k=\sum_{i=1}^n\varphi_{k,i}^{-1}(z_{1,i},\ldots,z_{m,i})Q_{i}
\label{manifold_inverse_spectral_resolution}
\end{equation}
for each $1\leq k\leq d$, where each 
$\varphi_{j,i}\in \mathcal{C}(Z^d)$ and where
\begin{equation}
\varphi_{j,i}(\varphi_{1,i}^{-1}(z_{1,i},\ldots,z_{m,i}),\ldots,\varphi_{d,i}^{-1}(z_{1,i},\ldots,z_{m,i}))=z_{j,i}
\label{local_matrix_manifold_structure}
\end{equation}
for each $1\leq i\leq n$ and each $1\leq j\leq m$. By the previously described facts and by \eqref{manifold_inverse_spectral_resolution} we will have 
that $\mathbf{H}:=(H_1,\dots,H_d)$ can be chosen from $\mathbb{I}^d(n)$, and that we can apply Lemma \ref{Existence_of_mD_PMA} to find a $\nu_\delta$-PMA 
$\tilde{\mathbf{H}}_{\delta}:=(\tilde{H}_{\delta,1},\ldots,\tilde{H}_{\delta,d})$ of 
$\mathbf{H}$ in $\mathbb{I}^m(n)$, where each $\tilde{H}_{\delta,k}$ can be decomposed in the form $\tilde{H}_{\delta,k}=\sum_j x_{k,j}P_j$ described in Lemma \ref{Existence_of_mD_PMA} for some $\nu_\delta$-projective decomposition $\{P_j\}=P_\delta(\mathbf{H})$. Because of the smooth structure of $Z^d$ and by 
\eqref{generic_matrix_neighborhoods}, \eqref{manifold_spectral_resolution}, \eqref{manifold_inverse_spectral_resolution} and \eqref{local_matrix_manifold_structure} we will have that for each $1\leq k\leq m$, there are $N_\delta$ (not necessarily distinct) differentiable charts $\varphi_{k,1},\ldots,\varphi_{k,N_\delta}$ counted with multiplicity in $\mathcal{C}(Z^d)$ such that for each $1\leq k\leq m$ the following relations hold.
\begin{equation}
\left\{
\begin{array}{l}
Z_k=\sum_{j=1}^{N_\delta} \varphi_{k,j}(P_j H_1,\ldots,P_j H_d)P_j,\\
\max_{1\leq j\leq N_{\delta}} \|\varphi_{k,j}(P_jH_1,\ldots,P_jH_d)P_j-\varphi_{k,j}(x_{1,j},\ldots,x_{d,j}) P_j\|\leq\delta
\end{array}
\right.
\label{local_manifold_approximation_inequality}
\end{equation}
If for each $k$ we set
\begin{equation}
Z_{k,\delta}:=\sum_{j=1}^{N_\delta} \varphi_{k,j}(x_{1,j},\ldots,x_{d,j}) P_j=\sum_{j=1}^{N_\delta}\varphi_{k,j}(P_j\tilde{H}_{\delta,1},\ldots,P_j\tilde{H}_{\delta,d}) P_j,
\label{multivariable_functional_calculus_identity}
\end{equation}
then the $m$-tuple $\tilde{\mathbf{Z}}_\delta:=(Z_{1,\delta},\ldots,Z_{m,\delta})$ belongs to $\mathbb{V}^m[Z^d](n)$ and 
the commutation relations
\begin{equation}
[Z_{k},Z_{l,\delta}]=\mathbf{0}_n,
\label{matrix_manifold_commutation_relations}
\end{equation}
hold for each $1\leq k,l\leq m$. By \eqref{local_matrix_manifold_structure}, \eqref{local_manifold_approximation_inequality} and \eqref{multivariable_functional_calculus_identity} we have that the relations \eqref{additional_matrix_manifold_conditions} hold for each $1\leq k\leq m$. By \eqref{generic_matrix_neighborhoods} and \eqref{multivariable_functional_calculus_identity} we will have that the following estimates hold for each $1\leq k\leq m$.
\begin{eqnarray}
\|Z_k-Z_{k,\delta}\|&=&\|(\sum_{j=1}^{N_\delta} P_j)(Z_k-Z_{k,\delta}))\|\nonumber\\
&=&\|\sum_{j=1}^{N_\delta} (\varphi_{k,j}(P_jH_1,\ldots,P_jH_d)P_j-\varphi_{k,j}(x_{1,j},\ldots,x_{d,j}) P_j)\|\nonumber\\
&=&\max_{1\leq j\leq N_{\delta}} \|\varphi_{k,j}(P_jH_1,\ldots,P_jH_d)P_j-\varphi_{k,j}(x_{1,j},\ldots,x_{d,j}) P_j\|\nonumber\\
&\leq&\delta. 
\label{smooth_matrix_manifold_approximation_inequality}
\end{eqnarray}
By \eqref{multivariable_functional_calculus_identity}, \eqref{matrix_manifold_commutation_relations} and \eqref{smooth_matrix_manifold_approximation_inequality} we will have that $\tilde{\mathbf{Z}}_\delta$ is a hermitian $\delta$-CPMA of 
$\mathbf{Z}$. This completes the proof.
 \end{proof}

\begin{definition}[Uniform piecewise differentiable local connectivity in $M_n^m$] If for any given $\varepsilon>0$, there is $\delta>0$ such that for any two $m$-tuples $\mathbf{X}$ and $\mathbf{Y}$ in a matrix variety $\mathbb{Z}^m\subset M_n^m$ such that 
$\eth(\mathbf{X},\mathbf{Y})\leq \delta$ ($\mathbf{X}$ and $\mathbf{Y}$ are $(\delta_\varepsilon,\eth)$-close), we have that $\mathbf{X}\leadsto_{\varepsilon(\delta)} \mathbf{Y}$ relative to $N_\eth(\mathbf{X},\varepsilon)\cap \mathbb{Z}^m$.
We will say that the matrix variety $\mathbb{Z}^m \subseteq M_n^m$ is uniformly locally piecewise differentiably connected ({\bf ULPDC}).
\end{definition}

\begin{remark}
It is important to remark that $\varepsilon$ and $\delta$ must not depend on $n$.
\end{remark}

\begin{remark}
It is important to notice that the components of two $m$-tuples $\mathbf{X}=(X_1,\ldots,X_m)$ and $\mathbf{Y}=(Y_1,\ldots,Y_m)$ in $\mathbb{I}^m(n)$ need not to satisfy the commutation relations $[X_j,Y_k]=\mathbf{0}_n$ for each $j,k$ in general.
\end{remark}

\subsubsection{Approximate Joint Isospectral Interpolating Paths} 
\label{approximate_joint_isospectral_flows}
Let $\varepsilon>0$ and let $m\geq 1$ be an integer, in this section we will prove that there is 
$\delta>0$ such that for any integer $n\geq 1$, any $m$-tuple $\mathbf{H}=(H_1,\ldots,H_m)\in \mathbb{I}^m(n)$ and any unitary matrix 
$W$ in $M_n$ that satisfy the relations $\|WH_j-H_jW\|<\delta$, $1\leq j\leq m$. There exist a unitary $\tilde{W}$ and a $\delta$-CPMA $\tilde{\mathbf{H}}=(\tilde{H}_1,\ldots,\tilde{H}_m)\in C^\ast(H_1,\ldots,H_m)'\cap \mathbb{I}^m(n)$ such that $W^\ast \tilde{H}_j W=\tilde{W}^\ast \tilde{H}_j\tilde{W}$ for each $1\leq j\leq m$. Moreover, there is a piecewise differentiable unitary path $\{W(t)\}_{t\in [0,1]}\subset \mathbb{U}(n)$ such that $W(0)=\mathbf{1}_n$, $W(1)=\tilde{W}$ and $\|W(t)\tilde{H}_j-\tilde{H}_jW(t)\|<\varepsilon$ for each $1\leq j\leq n$ and each $0\leq t\leq 1$. 

\begin{remark}
\label{aproximate_joint_isospectral_remark}
It can be seen that if we set $\Psi=\widehat{\mathrm{Ad}[W]}$, $\tilde{\Psi}=\widehat{\mathrm{Ad}[\tilde{W}]}$ and $\psi_t=\widehat{\mathrm{Ad}[W(t)]}$, we will have that.
\begin{equation}
\left\{
\begin{array}{l}
\psi_0(\tilde{\mathbf{H}})=\mathrm{id}_{\mathbb{I}^m(n)}(\tilde{\mathbf{H}})=\tilde{\mathbf{H}},\\
\psi_1(\tilde{\mathbf{H}})=\Psi(\tilde{\mathbf{H}})=\tilde{\Psi}(\tilde{\mathbf{H}}),\\
\psi_t(\tilde{\mathbf{H}})[j]\psi_t(\tilde{\mathbf{H}})[k]=\psi_t(\tilde{\mathbf{H}})[k]\psi_t(\tilde{\mathbf{H}})[j], \:\: 1\leq j,k\leq m, 0\leq t\leq 1,\\
\psi_t(\tilde{\mathbf{H}})\in B_\eth(\tilde{\mathbf{H}},\varepsilon), \:\: 0\leq t\leq 1,\\
\Lambda(\psi_t(\tilde{\mathbf{H}}))=\Lambda(\tilde{\mathbf{H}}), \:\: 0\leq t\leq 1,
\end{array}
\right.
\label{approx_isospectral}
\end{equation}
Any path of inner $\ast$-automorphisms like $\{\psi_t\}_{t\in [0,1]}$, that satisifes \eqref{approx_isospectral} will be called a {\bf $\varepsilon$-approximate joint isospectral interpolant} 
with respect to $(\tilde{\Psi},\tilde{\mathbf{H}},\mathbb{I}^m(n))$.
\end{remark}

\begin{lemma}
\label{existence_of_almost_unit}
 Given $r\in \ZZ^+$ and $\nu>0$, there is $\delta:=\hat{\delta}(\nu,r)>0$ such that given a unitary $W$ and a normal contraction $D$ in $M_n$ with $n\geq 2$, if $D=\sum_{j=1}^r \alpha_jP_j$ for $2\leq r\in \ZZ$ with $\NS{\alpha}{r}\in \disk$, and the 
 set $\{P_j\}$ is an orthogonal partition of unity consisting of diagonal projectors, then there 
 is a unitary matrix $Z\in M_n$ such that $[Z,D]=\mathbf{0}_n$ and $\|\I_n-WZ\|\leq\nu$ whenever $\|WDW^*-D\|\leq \delta$.
\end{lemma}
\begin{proof}
Let $\mu:=\|WDW^*-D\|=\|WD-DW\|$. Then there are $r$ continuous functions $\ell_1,\ldots,\ell_r\in C(\disk)$ such that 
$P_k:=\ell_k(D)$. By Lemma \ref{cont_funct_calc_tool} one can find $\delta:=\hat{\delta}(\nu,r)>0$ such that, $\|WP_kW^\ast-P_k\|=\|W\ell_k(D)W^\ast-\ell_k(D)\|= \|\ell_k(WDW^\ast)-\ell_k(D)\|\leq \nu/(\sqrt{2}r)<1/(\sqrt{2}r)$ whenever $\mu\leq \delta$. Since $\nu<1$, by perturbation theory of amenable $C^*$-algebras (in the sense of \cite[Lemma 2.5.1]{Amenable_algebras_Lin}) we will have that there is a unitary $W_j$ such that $\|\mathbf{1}-W_j\|\leq \sqrt{2}\|WP_jW^\ast-P_j\|\leq \nu/r$ and $W_j^\ast P_j W_j=WP_jW^\ast$, and this implies that $W_jWP_j=P_jW_j W$ for each $j$. Since $\{P_j\}$ is an orthogonal partition of unity, we will have that $\tilde{W}:=\sum_j W_jWP_j$ is a 
unitary matrix which satisfies the commutation relation $[\tilde{W},D]=0$ together with the normed inequalities.
\begin{eqnarray*}
\|W-\tilde{W}\|&=&\|W\sum_jP_j -\sum_j W_jWP_j\|\\
               &=&\|\sum_j((\mathbf{1}_n-W_j)WP_j)\|\\
               &\leq&\sum_j \|\mathbf{1}_n-W_j\|\\
               &\leq&\sum_j \frac{1}{r}\nu=r(\frac{1}{r}\nu)=\nu
\end{eqnarray*}
Let us set $Z:=\tilde{W}^\ast$ then $\|\mathbf{1}_n-WZ\|=\|\tilde{W}-W\|\leq \nu$. This completes the proof.
\end{proof}

\begin{lemma}
\label{existence_jpia}
Given $\nu>0$ and any $m\in \mathbb{Z}^ +$, there is $\delta>0$ such that for any two $m$-tuples $\mathbf{X},\mathbf{Y}\in \mathbb{I}^m(n)$ that satisfy the constraint $\eth(\mathbf{X},\mathbf{Y})<\delta$, there exist two inner morphisms $\Psi$ and $\tilde{\Psi}$ of $\mathbb{I}^m(n)$, together with a hermitian $\delta$-CPMA $\tilde{\mathbf{X}}\in \mathbb{I}^m(n)$ of $\mathbf{X}$ such that $\tilde{\Psi}(\tilde{\mathbf{X}})=\Psi(\tilde{\mathbf{X}})$, $\tilde{\Psi}(\mathbf{\tilde{X}})$ is a hermitian $\delta$-CPMA of $\Psi(\mathbf{X})$, $\Psi(\mathbf{X})$ is a hermitian $\nu$-CPMA of $\mathbf{Y}$ and $\eth(\tilde{\mathbf{X}},\tilde{\Psi}(\tilde{\mathbf{X}}))<\nu$. Moreover, there is a $\nu$-approximate joint isospectral interpolant $\{\psi_t\}_{t\in[0,1]}$
with respect to $(\tilde{\Psi},\tilde{\mathbf{X}},\mathbb{I}^m(n))$.
\end{lemma}
\begin{proof}
Let us consider two $m$-tuples $\mathbf{X}$ and $\mathbf{Y}$ in $\mathbb{I}^m(n)$, and let us set 
$\mu:=\eth(\mathbf{X},\mathbf{Y})$. We can assume, by changing basis if necessary, that $\mathbf{Y}$ is a $m$-tuple of diagonal matrices. By Lemma \ref{Joint_spectral_variation_inequality_2} we will have that there exists a $C^*$-homomorphism $\Phi:M_n\to M_n$ defined by the expression $\Phi=\mathrm{Ad}[W]$ for some $W\in \U{n}$, such that $\hat{\Phi}(\mathbf{X})$ is a $m$-tuple of diagonal contractions and 
\begin{equation}
\eth(\hat{\Phi}(\mathbf{X}),\mathbf{Y})\leq K_m\mu.
\label{special_joint_spectral_variation_inequality}
\end{equation}
By applying Lemma \ref{Existence_of_mD_PMA} to the hermitian $m$-tuple
\begin{equation}
\hat{\Phi}(\mathbf{X})=(\hat{Y}_{1},\ldots,\hat{Y}_{m}),
\label{Phi_hermitian_partition}
\end{equation}
we can obtain a hermitian $\mu$-CPMA 
\begin{equation}
\tilde{\mathbf{Y}}=(\tilde{Y}_{1},\ldots,\tilde{Y}_{m})
\label{PMA_Phi_hermitian_partition}
\end{equation}
of $\hat{\Phi}(\mathbf{X})$. We will have that the $m$-tuple $\tilde{\mathbf{X}}=\hat{\Phi}^{-1}(\tilde{\mathbf{Y}})$ satisfies the constraint $\eth(\mathbf{X},\tilde{\mathbf{X}})\leq \mu$.

As a consequence of Lemma \ref{Existence_of_mD_PMA} there exists a projective decomposition $P_\delta(\hat{\Phi}(\mathbf{X}))$ of the form
\begin{eqnarray}
P_\delta(\hat{\Phi}(\mathbf{X}))&=&\{P_{1},\ldots,P_{N_\delta}\}\nonumber \\
&=&\left\{P_{1,j_{1}}P_{2,j_{2}}\cdots P_{m,j_{m}}\left| \begin{array}{l}
P_{k,j_{k}}\in P_\delta(\hat{Y}_{k}),\\
1\leq j_{k}\leq N_{\delta,k},\\
1\leq k \leq m,
\end{array}
 \right.\right\},
 \label{def_projective_relations}
\end{eqnarray}
where each $\hat{Y}_{k}$ is the $k$th component of $m$-tuple \eqref{Phi_hermitian_partition}. As a consequence of the proof of Lemma \ref{Existence_of_mD_PMA}, we have that for each $1\leq k\leq m$ and each $1\leq j_{k}\leq N_{\delta}$, there exists a continuous function $p_{k,j_{k}}\in C([-1,1])$ such that 
\begin{equation}
P_{k,j_{k}}=p_{k,j_{k}}(\tilde{Y}_{k}),
\label{continuous_spectral_projectors}
\end{equation}
where each $\tilde{Y}_{k}$ is determined by \eqref{PMA_Phi_hermitian_partition}. By Lemma \ref{cont_funct_calc_tool} we will have that given $0<\nu_1<1$, there is $\delta>0$ such that
\begin{eqnarray}
\|\Phi^{-1}(P_{k,j_{k}})-P_{k,j_{k}}\|&=&\|\Phi^{-1}(p_{k,j_{k}}(\tilde{Y}_{k}))-p_{k,j_{k}}(\tilde{Y}_{k})\|\nonumber \\
&=&\|W^\ast p_{k,j_{k}}(\tilde{Y}_{k})W-p_{k,j_{k}}(\tilde{Y}_{k})\|\nonumber \\
&=&\|p_{k,j_{k}}(W^\ast \tilde{Y}_{k}W) -p_{k,j_{k}}(\tilde{Y}_{k})\|\nonumber\\
&<&\nu_1/{(2\sqrt{2}mN_{\delta})}<1/{(2\sqrt{2}mN_{\delta})}
\label{first_projective_inequality}
\end{eqnarray}
whenever $\mu<\delta$. As a consequence of \eqref{first_projective_inequality} we will have that for each projection $P_{j}=P_{1,j_{1}}\cdots P_{m,j_{m}}$ in $P_\delta(\hat{\Phi}(\mathbf{X}))$ determined by the relations \eqref{def_projective_relations}, the following estimate will hold.
\begin{eqnarray}
\|W^\ast P_{j} W-P_{j}\|&=&\|\Phi^{-1}(P_{j})-P_{j}\|\nonumber\\
&\leq& \sum_{k=1}^m\|\Phi^{-1}(P_{k,j_{k}})-P_{k,j_{k}}\|\nonumber\\
&<& m\nu_1/{(2m\sqrt{2}N_{\delta}})=\nu_1/{(2\sqrt{2}N_{\delta})}
\label{second_projective_inequality}
\end{eqnarray}
Since the inequalities \eqref{second_projective_inequality} hold for each $P_{j}$ in $P_\delta(\hat{\Phi}(\mathbf{X}))$, we can apply Lemma \ref{existence_of_almost_unit} with $r=N_{\delta}$ to compute a unitary $\tilde{W}\in M_n$ that satisfies the constraint 
\begin{equation}
\|\I_n-\tilde{W}^\ast W\|=\|W-\tilde{W}\|\leq \nu_1/2,
\label{almost_unit_projective_inequality}
\end{equation}
together with the commutation relations
\begin{equation}
\tilde{W}P_{j}=P_{j}\tilde{W}
\label{projective_commutation_relations}
\end{equation}
for each $P_{j}$ in $P_\delta(\hat{\Phi}(\mathbf{X}))$. Let us set $\tilde{\Phi}=\mathrm{Ad}[Z]$ with 
$Z=\tilde{W}^\ast W$. By Lemma \ref{Existence_of_mD_PMA} we have that each $\tilde{Y}_{j}$ determined by \eqref{PMA_Phi_hermitian_partition} belongs to $\mathrm{span}\: P_\delta(\hat{\Phi}(\mathbf{X}))$, as a consequence of \eqref{projective_commutation_relations} we will have that
\begin{equation}
\tilde{\Phi}(\Phi^{-1}(\tilde{Y}_{j}))=\tilde{W}^\ast \tilde{Y}_{j}\tilde{W}=\tilde{Y}_{j}=\Phi(\Phi^{-1}(\tilde{Y}_{j})),
\label{first_local_indentity_relation}
\end{equation}
by applying Lemma \ref{Joint_spectral_variation_inequality_2} again together with the commutation relations \eqref{projective_commutation_relations} we will also have that
\begin{eqnarray}
\|\tilde{\Phi}(\Phi^{-1}(\tilde{Y}_{j}))-\Phi^{-1}(\tilde{Y}_{j})\|&=&\|\tilde{W}^\ast \tilde{Y}_{j}\tilde{W}-W^\ast \tilde{Y}_{j}W\|\nonumber\\
&=&\|\tilde{Y}_{j}-W^\ast \tilde{Y}_{j}W\|\nonumber\\
&\leq&\|\tilde{Y}_{j}-\hat{Y}_{j}\|+\|\hat{Y}_{j}-W^\ast \hat{Y}_{j}W\|\nonumber\\ 
&&+\|W^\ast \hat{Y}_{j}W-W^\ast \tilde{Y}_{j}W\|\nonumber\\
&\leq&2\|\tilde{Y}_{j}-\hat{Y}_{j}\|+\|\Phi(X_j)-X_j\| \nonumber\\
&<&2\delta+K_m\delta=(K_m+2)\delta.
\label{second_local_indentity_relation}
\end{eqnarray}
Since $\|\I_n-Z\|<\nu_1/2<1$ by \eqref{almost_unit_projective_inequality}, we will have that there is a hermitian matrix contraction $H\in M_n$ such that $Z=e^{\mathbf{i}\pi H/2}$. If for each $t\in [0,1]$ we set 
$\phi_t:=\mathrm{Ad}[e^{\mathbf{i}\pi t H/2}]$, we will have that $\phi_t:M_n\to M_n$ is a $\ast$-homomorphism for each $0\leq t\leq 1$ such that $\phi_0=\mathrm{id}_{M_n}$ and $\phi_1=\tilde{\Phi}$. We will also have that for each 
$0\leq t\leq 1$ and each $\tilde{Y}_{j}$ in \eqref{PMA_Phi_hermitian_partition}
\begin{equation}
\|\phi_t(\Phi^{-1}(\tilde{Y}_{j}))-\Phi^{-1}(\tilde{Y}_{j})\|\leq 2\|\I_n-e^{\mathbf{i}\pi t H/2}\|\leq 2\|\I_n-e^{\mathbf{i}\pi H/2}\|<\nu_1. 
\label{component_path_inequality}
\end{equation}
Let us set $\nu:=\max\{\nu_1,(K_m+2)\delta\}$, $\Psi=\hat{\Phi}$, $\tilde{\Psi}:=\hat{\tilde{\Phi}}$ and $\psi_t:=\hat{\phi_t}$ for each $t\in [0,1]$. By Remark \ref{inner_morphisms_remark} we have that the maps $\Psi$, $\tilde{\Psi}$ and all members of the family $\{\psi_t\}_{t\in[0,1]}$ preserve commutativity, hermiticity and the spectrum of each component $Z_j$ of any $m$-tuple $\mathbf{Z}\in \mathbb{I}^m(n)$, which implies that all of them belong to the set of morphisms of $\mathbb{I}^m(n)$. By the relations \eqref{first_local_indentity_relation} we have that.
\begin{eqnarray}
\tilde{\Psi}(\tilde{\mathbf{X}})&=&(\tilde{\Phi}(\Phi^{-1}(\tilde{Y}_{1})),\ldots,\tilde{\Phi}(\Phi^{-1}(\tilde{Y}_{m})))\nonumber\\
&=&(\Phi(\Phi^{-1}(\tilde{Y}_{1})),\ldots,\Phi(\Phi^{-1}(\tilde{Y}_{m})))\nonumber\\
&=&\Psi(\tilde{\mathbf{X}}).
\label{local_representation_identity_of_Psi}
\end{eqnarray} 
By the properties of \eqref{PMA_Phi_hermitian_partition}, by relation \eqref{local_representation_identity_of_Psi} and by inequalities \eqref{special_joint_spectral_variation_inequality} and \eqref{second_local_indentity_relation} we also have that
\begin{equation}
\left\{
\begin{array}{l}
\eth(\mathbf{X},\tilde{\mathbf{X}})=\eth(\Psi(\mathbf{X}),\tilde{\Psi}(\tilde{\mathbf{X}}))=\eth(\Psi(\mathbf{X}),\Psi(\tilde{\mathbf{X}}))< \delta,\\
\eth(\mathbf{Y},\Psi({\mathbf{X}}))=\eth(\mathbf{Y},\hat{\Phi}({\mathbf{X}}))<\nu\\ 
\eth(\tilde{\mathbf{X}},\tilde{\Psi}(\tilde{\mathbf{X}}))\leq\max_{j} \|\tilde{\Phi}(\Phi^{-1}(\tilde{Y}_{j}))-\Phi^{-1}(\tilde{Y}_{j})\|<\nu,
\end{array}
\right.
\label{first_conclusion_inequalities}
\end{equation}
and by combining \eqref{local_representation_identity_of_Psi} and \eqref{first_conclusion_inequalities} we will have that $\tilde{\Psi}(\tilde{\mathbf{X}})=\Psi(\tilde{\mathbf{X}})$, $\tilde{\Psi}(\mathbf{\tilde{X}})$ is a hermitian $\delta$-CPMA of $\Psi(\mathbf{X})$, $\Psi(\mathbf{X})$ is a hermitian $\nu$-CPMA of $\mathbf{Y}$ and $\eth(\tilde{\mathbf{X}},\tilde{\Psi}(\tilde{\mathbf{X}}))<\nu$. By definition of the elements in the family $\{\psi_t\}_{t\in [0,1]}$ we will have that $\psi_0=\widehat{\mathrm{id}_{M_n}}=\mathrm{id}_{\mathbb{I}^m(n)}$, 
$\psi_1=\widehat{\mathrm{Ad}[Z]}=\tilde{\Psi}$ and for each $t\in [0,1]$ by \eqref{component_path_inequality} we will have that
\begin{equation*}
\eth(\psi_t(\tilde{\mathbf{X}}),\tilde{\mathbf{X}})\leq \max_{j} \|\phi_t(\Phi^{-1}(\tilde{Y}_{j}))-\Phi^{-1}(\tilde{Y}_{j})\|<\nu.
\end{equation*}
This completes the proof.
\end{proof}

\begin{theorem}
\label{main_result}
The matrix variety $\mathbb{I}^m(n)$ is uniformly locally piecewise differentiably connected.
\end{theorem}
\begin{proof}
As a consequence of Lemma \ref{existence_jpia} we will have that given $\nu>0$, there is $\delta>0$, such that for any two $m$-tuples $\mathbf{X},\mathbf{Y}\in \mathbb{I}^m(n)$ that satisfy the constraint $\eth(\mathbf{X},\mathbf{Y})< \delta$, there exist a $m$-tuple $\tilde{X}\in \mathbb{I}^m(n)$, two inner morphisms 
$\Psi$ and $\tilde{\Psi}$ of $\mathbb{I}^m(n)$, and a family $\{\psi_t\}_{t\in[0,1]}$ of morphisms of $\mathbb{I}^m(n)$ such that $\tilde{\Psi}(\tilde{\mathbf{X}})=\Psi(\tilde{\mathbf{X}})$, $\tilde{\Psi}(\mathbf{\tilde{X}})$ is a hermitian $\delta$-CPMA of $\Psi(\mathbf{X})$, $\Psi(\mathbf{X})$ is a hermitian $\nu$-CPMA of $\mathbf{Y}$ and 
\begin{equation}
\eth(\tilde{\mathbf{X}},\tilde{\Psi}(\tilde{\mathbf{X}}))<\nu.
\label{controlled_curved_inequality}
\end{equation}
Moreover, the family of morphisms $\{\psi_t\}_{t\in[0,1]}$ of $\mathbb{D}^m(n)$ satisfies the following conditions.
\begin{equation}
\left\{
\begin{array}{l}
\psi_0=\mathrm{id}_{\mathbb{I}^m(n)},\: \: \psi_1=\tilde{\Psi},\\
\eth(\tilde{\mathbf{X}},\psi_t(\tilde{\mathbf{X}}))<\nu, 0\leq t\leq 1
\end{array}
\right.
\label{curved_path_conditions}
\end{equation}
Given any two $m$-tuples in $\mathbf{H}=(H_1,\ldots,H_m)$ and $\mathbf{K}=(K_1,\ldots,K_m)$ in $\mathbb{I}^m(n)$, let us write $\ell^H_{\mathbf{H},\mathbf{K}}$ to denote the path in $C([0,1],M_n^m)$ defined by the expression
\begin{equation}
\ell^H_{\mathbf{H},\mathbf{K}}(t)=((1-t)H_1+tK_1,\ldots,(1-t)H_m+tK_m), \:\: 0\leq t\leq 1.
\label{hermitian_line_definition}
\end{equation}
It can be seen that $(\ell^H_{j,\mathbf{H},\mathbf{K}}(t))^\ast=((1-t)H_j+tK_j)^\ast=(1-t)H_j+tK_j=\ell^H_{j,\mathbf{H},\mathbf{K}}(t)$ for each $1\leq j\leq m$ and each $0\leq t\leq 1$. It can also be seen that $\ell^H_{\mathbf{H},\mathbf{K}}(0)=\mathbf{H}$, $\ell^H_{\mathbf{H},\mathbf{K}}(1)=\mathbf{K}$, $\eth(\ell^H_{\mathbf{H},\mathbf{K}}(t),\ell^H_{\mathbf{H},\mathbf{K}}(s))\leq \eth(\mathbf{H},\mathbf{K})$ for any $0\leq s,t\leq 1$, and for each $1\leq j\leq m$ we have that $\|(1-t)H_j+tK_j\|\leq (1-t)\|H_j\|+t\|K_j\|\leq 1$, for each $0\leq t\leq 1$. By the previous facts we will have that if $H_iK_j=K_j H_i$ for each $1\leq i,j\leq m$, then $\ell^H_{\mathbf{H},\mathbf{K}}(t)\in \mathbb{I}^m(n)$ for each $0\leq t\leq 1$. Let us write $\kappa$ to denote the path defined by the following expression
\begin{equation}
\kappa(t)=\psi_t(\tilde{\mathbf{X}}), \:\: 0\leq t\leq 1.
\label{particular_curved_path_definition}
\end{equation}
By Lemma \ref{existence_jpia} and by \eqref{curved_path_conditions} we will have that $\kappa(0)=\tilde{\mathbf{X}}$, $\kappa(1)=\tilde{\Psi}(\tilde{\mathbf{X}})$, $\kappa(t)\in \mathbb{I}^m(n)$ and $\eth(\kappa(t),\tilde{\mathbf{X}})=\eth(\kappa(t),\kappa(0))<\nu$ for each $t\in [0,1]$.

Let us set $\varepsilon:=2(\nu+\delta)$. By the arguments presented in the previous paragraphs, we will have that the piecewise differentiable path $\gamma\in C([0,1],M_n^m)$ 
defined by the expression
\begin{equation*}
\gamma=((\ell^H_{\mathbf{X},\tilde{\mathbf{X}}}\circledast \kappa)\circledast  \ell^H_{\tilde{\Psi}(\tilde{\mathbf{X}}),\Psi(\mathbf{X})})\circledast \ell^H_{\Psi(\mathbf{X}),\mathbf{Y}},
\label{definition_of_solvent_path}
\end{equation*}
satisfies the conditions
\begin{equation}
\left\{
\begin{array}{l}
\gamma(0)=\mathbf{X}, \:\: \gamma(1)=\mathbf{Y},\\
\gamma(t)\in \mathbb{I}^m(n), \:\: 0\leq t\leq 1,\\
\eth(\gamma(t),\mathbf{X})<\varepsilon, \:\: 0\leq t\leq 1.
\end{array}
\right.
\label{local_path_connectivity_conditions}
\end{equation}
By \eqref{local_path_connectivity_conditions} we have that $\mathbf{X}\leadsto_{\varepsilon(\delta)} \mathbf{Y}$ relative to $N_\eth(\mathbf{X},\varepsilon)\cap \mathbb{I}^m(n)$. This completes the proof.
\end{proof}

\begin{corollary}
\label{main_corollary}
The matrix variety $\mathbb{D}^m(n)$ is uniformly locally piecewise differentiably connected.
\end{corollary}
\begin{proof}
Given any two $m$-tuples $\mathbf{Z}$ and $\mathbf{S}$ in $\mathbb{D}^m(n)$, such that $\eth(\mathbf{Z},\mathbf{S})\leq r$ for some 
$r>0$, there are two semisimple commuting hermitian partitions $\pi(\mathbf{Z})$ and $\pi(\mathbf{S})$ of $\mathbf{Z}$ and $\mathbf{S}$ respectively, such that $\eth(\pi(\mathbf{Z}),\pi(\mathbf{S}))\leq r$. By the previously described fact, we can apply Theorem \ref{main_result} to the corresponding  semisimple commuting hermitian partitions of any two $(\delta_\varepsilon,\eth)$-close $m$-tuples $\mathbf{Z},\mathbf{S}$ in $\mathbb{D}^m(n)$, replacing $\delta_\varepsilon$ by $\delta_\varepsilon/2$ if necessary, to obtain a path $\gamma^H\in C([0,1],\mathbb{I}^m(n))$ that solves the problem $\pi(\mathbf{Z}) \rightsquigarrow_{\varepsilon(\delta)} \pi(\mathbf{S})$ relative to $\mathbb{I}^m(n)\cap B_\eth(\mathbf{\pi(Z)},\varepsilon)$. The path $\hat{\gamma}^D$ obtained by taking the juncture $\hat{\gamma}^D(t)=\upsilon(\gamma^H(t))$ for each $t\in [0,1]$, will solve the problem $\mathbf{Z} \rightsquigarrow_{\varepsilon(\delta)} \mathbf{S}$ relative to $\mathbb{D}^m(n)\cap B_\eth(\mathbf{Z},\varepsilon)$. This completes the proof.
\end{proof}

\begin{theorem}
\label{main_generic_result}
Given a compact differentiable $d$-manifold $Z^d\subseteq [-1,1]^m \subseteq \RR^m$, the induced compact differentiable matrix manifold $\mathbb{V}^m[Z^d](n)$ is uniformly locally piecewise differentiably connected.
\end{theorem}
\begin{proof}
As a consequence of Lemma \ref{Existence_of_mD_PMA_manifolds} and Remarks \ref{inner_morphisms_remark} and \ref{aproximate_joint_isospectral_remark}, by a similar argument to the one implemented in the proof of Lemma \ref{existence_jpia} we will have that given $\nu_1>0$, there is $\delta>0$, such that for any two $m$-tuples $\mathbf{U},\mathbf{V}\in \mathbb{V}^m[Z^d](n)$ that satisfy the constraint $\eth(\mathbf{U},\mathbf{V})\leq \delta$, there exist: a $\delta$-CPMA $\tilde{\mathbf{U}}\in \mathbb{V}^m[Z^d](n)$ of $\mathbf{U}$, two inner morphisms 
$\Psi$ and $\tilde{\Psi}$ of $\mathbb{V}^m[Z^d](n)$, and a family $\{\psi_t\}_{t\in[0,1]}$ of morphisms of $\mathbb{V}^m[Z^d](n)$ such that $\tilde{\Psi}(\tilde{\mathbf{U}})=\Psi(\tilde{\mathbf{U}})$, $\tilde{\Psi}(\mathbf{\tilde{U}})$ is a $\delta$-CPMA of $\Psi(\mathbf{U})$, $\Psi(\mathbf{U})$ is a normal $\nu_1$-CPMA of $\mathbf{V}$ and 
\begin{equation}
\eth(\tilde{\mathbf{U}},\tilde{\Psi}(\tilde{\mathbf{U}}))<\nu_1.
\label{unitary_controlled_curved_inequality}
\end{equation}
Moreover, the family of morphisms $\{\psi_t\}_{t\in[0,1]}$ of $\mathbb{D}^m(n)$ satisfies the following conditions.
\begin{equation}
\left\{
\begin{array}{l}
\psi_0=\mathrm{id}_{\mathbb{V}^m[Z^d](n)},\: \: \psi_1=\tilde{\Psi},\\
\eth(\tilde{\mathbf{U}},\psi_t(\tilde{\mathbf{U}}))<\nu_1, 0\leq t\leq 1\\
\psi_t(\tilde{\mathbf{U}})[k] \psi_t(\tilde{\mathbf{U}})[j]=\psi_t(\tilde{\mathbf{U}})[j] \psi_t(\tilde{\mathbf{U}})[k], 1\leq j,k\leq m\\
\Lambda(\psi_t(\tilde{\mathbf{U}}))\subseteq Z^d, 0\leq t\leq 1
\end{array}
\right.
\label{unitary_curved_path_conditions}
\end{equation}
Given any two $m$-tuples in $\mathbf{W}=(W_1,\ldots,W_m)$ and $\mathbf{Z}=(Z_1,\ldots,Z_m)$ in $\mathbb{V}^m[Z^d](n)$ such that $Z_jW_k=W_kZ_j$ for each $1\leq j,k\leq m$, let us set 
$r:=\eth(\mathbf{W},\mathbf{Z})$ and let us write $\ell^V_{\mathbf{W},\mathbf{Z}}$ to denote the path in $C([0,1],M_n^m)$ defined coordinate wise by the expressions:
\begin{equation}
\ell^{V,k,j}_{\mathbf{W},\mathbf{Z}}(t)=\varphi_{k,j}((1-t)H_{W,1}+tH_{Z,1},\ldots,(1-t)H_{W,d}+tH_{Z,d})P_j, \:\: 0\leq t\leq 1,
\label{unitary_line_definition}
\end{equation}
with $1\leq k\leq m$. Here for each $1\leq j\leq d$, $H_{W,j},H_{Z,j}$ are the hermitian matrices described by Lemma \ref{Existence_of_mD_PMA_manifolds}, which satisfy the relations $W_k=\sum_j\varphi_{k,j}(H_{W,1},\ldots,H_{W,d})P_j$, $Z_k=\sum_j\varphi_{k,j}(H_{Z,1},\ldots,H_{Z,d})P_j$ and $\eth(\mathbf{H}_{W},\mathbf{H}_{Z})<\nu(r)$, and $\nu$ is the function corresponding to the parametrization of the components of \eqref{unitary_line_definition}. It can be seen that 
\begin{equation}
\left\{
\begin{array}{l}
\ell^V_{\mathbf{W},\mathbf{Z}}(0)=\mathbf{W}, \ell^V_{\mathbf{H},\mathbf{K}}(1)=\mathbf{Z}\\
\eth(\ell^V_{\mathbf{W},\mathbf{Z}}(t),\ell^V_{\mathbf{W},\mathbf{Z}}(s))\leq \nu(\eth(\ell^H_{\mathbf{H}_{\mathbf{W}},\mathbf{H}_{\mathbf{Z}}}(t),\ell^H_{\mathbf{H}_{\mathbf{W}},\mathbf{H}_{\mathbf{Z}}}(s))), 0\leq s,t\leq 1,
\end{array}
\right.
\label{unitary_lines_conditions}
\end{equation}
with $\ell^H_{\mathbf{H}_{\mathbf{W}},\mathbf{H}_{\mathbf{Z}}}$ defined by \eqref{hermitian_line_definition}, and for each $0\leq t\leq 1$ we have that the $m$-tuple $(\ell^{V,1}_{\mathbf{W},\mathbf{Z}}(t),\ldots,\ell^{V,m}_{\mathbf{W},\mathbf{Z}}(t))$, with $\ell^{V,k}_{\mathbf{W},\mathbf{Z}}(t)=\sum_j\varphi_{k,j}((1-t)H_{W,1}+tH_{Z,1},\ldots,(1-t)H_{W,d}+tH_{Z,d})P_j$, $1\leq k\leq m$, belongs to $\mathbb{V}^m[Z^d](n)$. By the previous facts we will have that $\ell^V_{\mathbf{W},\mathbf{Z}}(t)\in \mathbb{V}^m[Z^d](n)$ for each $0\leq t\leq 1$. Let us write $\kappa$ to denote the path defined by the following expression
\begin{equation}
\kappa(t)=\psi_t(\tilde{\mathbf{U}}), \:\: 0\leq t\leq 1.
\label{unitary_particular_curved_path_definition}
\end{equation}
By Lemma \ref{existence_jpia} and \eqref{unitary_curved_path_conditions} we will have that $\kappa(0)=\tilde{\mathbf{U}}$, $\kappa(1)=\tilde{\Psi}(\tilde{\mathbf{U}})$, $\kappa(t)\in \mathbb{V}^m[Z^d](n)$ and $\eth(\kappa(t),\tilde{\mathbf{U}})=\eth(\kappa(t),\kappa(0))<\nu_1$ for each $t\in [0,1]$.

Let us set $\varepsilon:=2(\nu_1+\max\{\max_{x\in [0,\delta]}\nu(x),\delta\})$. By the arguments presented in the previous paragraphs, we will have that the path $\gamma\in C([0,1],M_n^m)$ 
defined by the expression
\begin{equation*}
\gamma=((\ell_{\mathbf{U},\tilde{\mathbf{U}}}\circledast \kappa)\circledast  \ell_{\tilde{\Psi}(\tilde{\mathbf{U}}),\Psi(\mathbf{U})})\circledast \ell_{\Psi(\mathbf{U}),\mathbf{V}},
\label{definition_of_unitary_solvent_path}
\end{equation*}
satisfies the conditions
\begin{equation}
\left\{
\begin{array}{l}
\gamma(0)=\mathbf{U}, \:\: \gamma(1)=\mathbf{V},\\
\gamma(t)\in \mathbb{V}^m[Z^d](n), \:\: 0\leq t\leq 1,\\
\eth(\gamma(t),\mathbf{U})<\varepsilon, \:\: 0\leq t\leq 1.
\end{array}
\right.
\label{local_unitary_path_connectivity_conditions}
\end{equation}
By \eqref{local_path_connectivity_conditions} we have that $\mathbf{U}\leadsto_{\varepsilon(\delta)} \mathbf{V}$ relative to $N_\eth(\mathbf{U},\varepsilon)\cap \mathbb{V}^m[Z^d](n)$. This completes the proof.
\end{proof}

\subsection{Graphical analogies of continuous matrix deformations}

Given a $d$-mainfold $Z^d\subseteq [-1,1]^m$, the effect of an approximate joint isospectral interpolant $\psi\in C([0,1],M_n^m)$ on a section of an arbitrary compact matrix manifold 
$\mathbb{V}^m[Z^d](n)\subseteq \mathbb{I}^m(n)$ in $M_n$ can be interpreted as a matrix representation 
of a continuous deformation of $Z^d$, like the one described in Figure \ref{sphere_isospectral_deformation} in the particular case when $Z^d=\mathbb{S}^2$.

\begin{figure}[!htb]
\centering
 \includegraphics[scale=0.5]{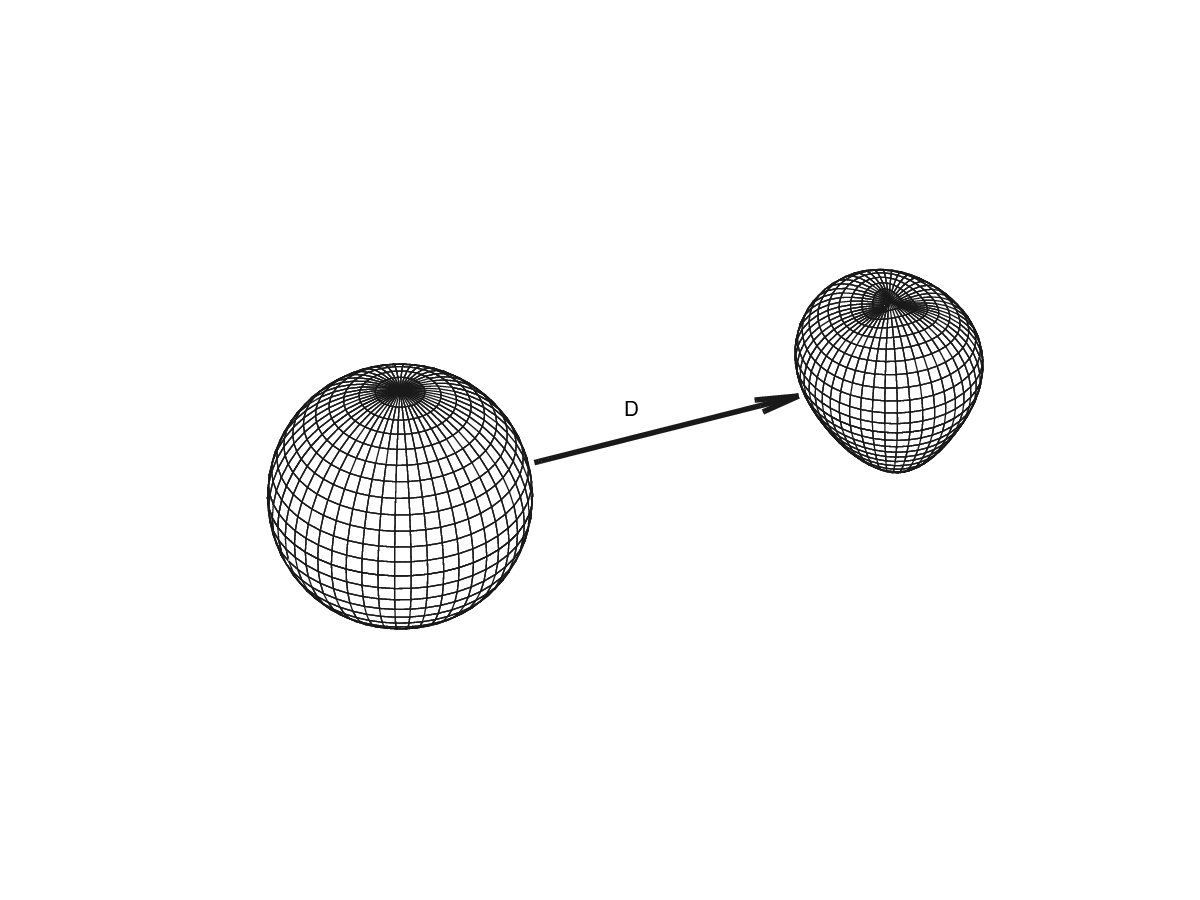}
 \caption{Graphical illustration of an approximate isospectral deformation in $V^3[\mathbb{S}^2](n)$.}
 \label{sphere_isospectral_deformation}
 \end{figure}
 
Given a fixed but arbitrary $3$-tuple $\mathbf{X}\in V^3[\mathbb{S}^2](n)\subseteq \mathbb{I}^3(n)$, we can think of the first homeomorphic copy of $\mathbb{S}^2$ on the left in Figure \ref{sphere_isospectral_deformation} as the section of $V^3[\mathbb{S}^2](n)$ corresponding to the joint spectra $\Lambda(\hat{\mathscr{D}}(V^3[\mathbb{S}^2](n)\cap C^\ast(\mathbf{X})))$ (the image of the full pinching operator extended to $3$-tuples in $V^3[\mathbb{S}^2](n)$) with respect to a basis in which $\mathbf{X}$ is a $m$-tuple of diagonal matrices, while the second homeomorphic copy of $\mathbb{S}^2$ on the right in Figure \ref{sphere_isospectral_deformation} can be interpreted as the section of $V^3[\mathbb{S}^2](n)$ corresponding to the joint spectra $\Lambda(\hat{\mathscr{D}}(V^3[\mathbb{S}^2](n)\cap C^\ast(\psi_1(\mathbf{X}))))$ on the same basis in which $\mathbf{X}$ is a $m$-tuple of diagonal matrices.

Let us consider two $m$-tuples $\mathbf{X}$ and $\mathbf{Y}$ in a certain matrix variety $Z^m\subseteq M_n^m$. We can think of each path $X_j\leadsto_{\varepsilon(\delta)} \tilde{X}_j$ relative 
to $N_\eth(X_j,\varepsilon)$, as a noncommutative analogy of the family of paths ({\bf pseudospectral-links}) connecting the point sets determined 
by the spectra $\sigma(\mathscr{D}(X_j))$ and 
$\sigma(\tilde{X}_j)$ relative to the embedding of $\disk\times [0,1]$ in $[-1,1]^2\times [0,1]$ (where the embedding is induced by the mapping $\disk\to B(0,1)\subseteq [-1,1]^2$). In a similar way we can interpret the induced paths 
$\mathbf{X}\leadsto_{\varepsilon(\delta)} \tilde{\mathbf{X}}$ as noncommutative analogies of the links connecting the joint spectra $\sigma(\mathbf{X})$ and 
$\sigma(\mathbf{\tilde{X}})$ relative to $[-1,1]^{2n}\times [0,1]$. An illustration of these analogies are presented in F.\ref{pseudospectral_links}, F.\ref{pseudospectral_links_clusters} and F.\ref{sphereical_pseudospectral_paths}.

\begin{figure}[!htb]
\centering
 \includegraphics[scale=0.4]{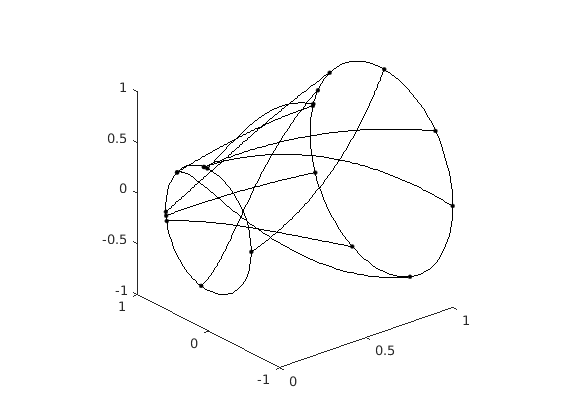}
 \caption{Graphical illustration of a the pseudospectral-links in $\disk\times [0,1]$ corresponding to a matrix path that connects the spectra of two normal contractions $X$ and $Y$ (with $Y$ diagonal) in $M_{10}$ with $\sigma(\mathscr{D}(X))\subseteq \partial B(0,3/5)$ and $\sigma(\mathscr{D}(Y))=\sigma(Y)\subseteq \TT$.}
 \label{pseudospectral_links}
 \end{figure}

\begin{figure}[!htb]
\centering
 \includegraphics[scale=0.4]{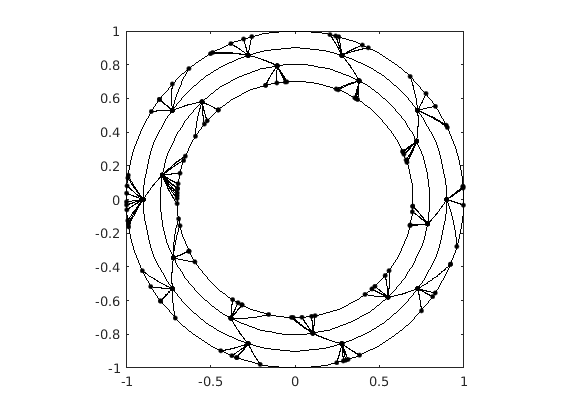}
 \caption{Graphical illustration of a the pseudospectral-links in $\disk$ corresponding to a matrix path that connects the spectra of two normal contractions $X$ and $Y$ (with $Y$ diagonal) in $M_{55}$ with $\sigma(\mathscr{D}(X))\subseteq \partial B(0,7/10)$ and $\sigma(\mathscr{D}(Y))=\sigma(Y)\subseteq \TT$.}
 \label{pseudospectral_links_clusters}
 \end{figure}
 
 \begin{figure}[!htb]
\centering
 \includegraphics[scale=0.4]{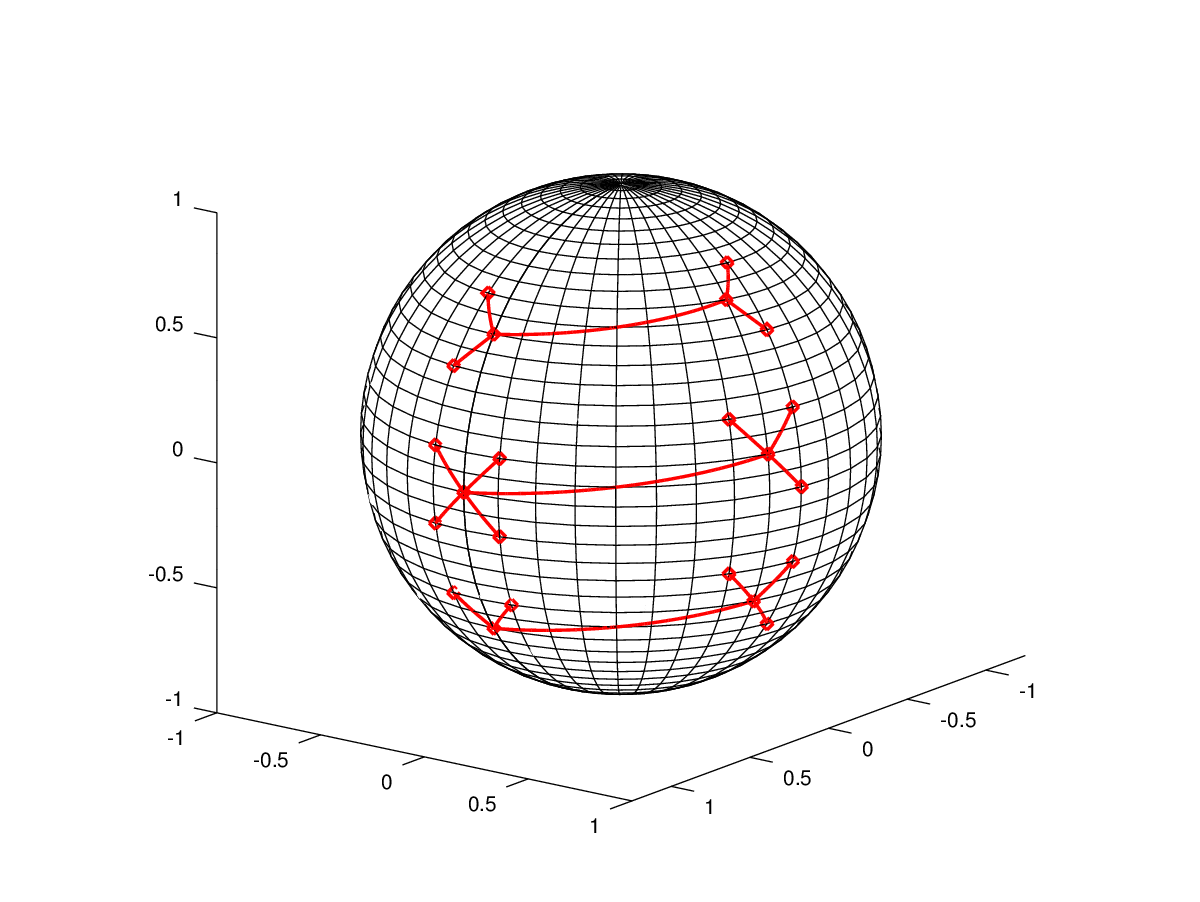}
 \caption{Graphical illustration of a the pseudospectral-paths in $\mathbb{S}^2$ corresponding to a matrix path that connects the spectra of two triples of pairwise commuting hermitian contractions $\mathbf{X}=(X_1,X_2,X_3)$ and $\mathbf{Y}=(Y_2,Y_2,Y_3)$ in $\mathbb{V}^3[\mathbb{S}^2](7)$.}
 \label{sphereical_pseudospectral_paths}
 \end{figure}
 
 The paths described in F.\ref{pseudospectral_links_clusters} and F.\ref{sphereical_pseudospectral_paths} provide a graphical illustration of the ideas behind the  {\em pseudospectral clustering} technique implemented in the proofs of Theorems \ref{main_result} and \ref{main_generic_result}.

 \pagebreak

\section{Applications}
\label{applications}

\subsection{Local Connectivity of some particular Compatc differentiable Matrix Manifolds}

Let us consider the semialgebraic matrix varieties $\mathbb{T}^m(n)$ and $\mathbb{S}^{m-1}(n)$ defined as follows.
\begin{equation}
\mathbb{T}^m(n)=\left\{\mathbf{U}\in \mathbb{D}^2(n)\:|\: U_k U_k^\ast=\mathbf{1}_n, 1\leq j,k\leq m\right\}
\label{definition_Tm_n}
\end{equation}

\begin{equation}
\mathbb{S}^{m-1}(n)=\left\{\mathbf{H}\in \mathbb{I}^{m}(n)\:\left|\: \sum_{k=1}^m H_k^2=\mathbf{1}_n\right.\right\}
\label{definition_Sm_n}
\end{equation}

\begin{lemma}
\label{secondary_corollary}
The matrix semialgebraic variety $\mathbb{T}^m(n)$ is {\bf ULPDC}.
\end{lemma}
\begin{proof}
Let $m\in \mathbb{Z}^+$. Since $\mathbb{T}^m(n)$ is defined by \eqref{definition_Tm_n}, we will have that if for any 
$\mathbf{U}\in\mathbb{T}^m(n)$ we set 
\begin{equation}
\pi_H(\mathbf{U})=(Re(U_1),Im(U_1),\ldots,Re(U_m),Im(U_m)),
\label{matrix_homeomorphism}
\end{equation}
then $\Lambda(\pi_H(\mathbf{U}))\subseteq \mathbb{T}^m=\Pi_{j=1}^m \mathbb{S}^1\subseteq \Pi_{j=1}^m [-1,1]^2\subseteq \mathbb{R}^{2m}$. By considering the mapping $\upsilon_H:M_n^{2m}\to M_n^m, (X_1,X_2,\ldots,X_{2m-1},X_{2m})\mapsto(X_1+\mathbf{i}X_2,\ldots,X_{2m-1}+\mathbf{i}X_{2m})$, it can be seen that the mapping $\pi_H:M_n^m\to M_n^{2m}$ determined by equation \eqref{matrix_homeomorphism} induces a one to one correspondence between $\mathbb{T}^m(n)$ and $V^{2m}(\mathbb{T}^m)(n)$. Because of this correspondence, we have that any $\mathbf{U}\in \mathbb{T}^m(n)$ has a representation of the form $\mathbf{U}=\upsilon_H(\pi_H(\mathbf{U}))$. 

Given  any two $(\nu,\eth)$-close tuples $\mathbf{U},\mathbf{V}\in \mathbb{T}^m(n)$ for some given $\nu>0$, we will have that the tuples $\pi_H(\mathbf{U}),\pi_H(\mathbf{V})\in V^{2m}[\mathbb{T}^m](n)$ will be $(\nu,\eth)$-close. By Theorem \ref{main_generic_result} we have that $V^{2m}(\mathbb{T}^m)(n)$ is {\bf ULPDC}, this implies that given $\varepsilon>0$, there is $\delta>0$ such that if $\nu<\delta$, there is a piecewise differentiable path $\gamma\in C([0,1],V^{2m}[\mathbb{T}^m](n))$ that solves the connectivity problem $\pi_H(\mathbf{U})\rightsquigarrow_{\varepsilon/2} \pi_H(\mathbf{V})$ relative to $V^{2m}[\mathbb{T}^m](n)\cap B_\eth(\pi_H(\mathbf{U}),\varepsilon/2)$. By the previous facts we will have that the path $\hat{\gamma}$ defined by the equation $\hat{\gamma}(t)=\upsilon_H(\gamma(t))$ for each $t\in [0,1]$, solves the connectivity problem $\mathbf{U}\rightsquigarrow_\varepsilon \mathbf{V}$ relative to $\mathbb{T}^m(n)\cap B_\eth(\mathbf{U},\varepsilon)$. This completes the proof.
\end{proof}

\begin{lemma}
\label{third_corollary}
The matrix semialgebraic variety $\mathbb{S}^{m-1}(n)$ is {\bf ULPDC}.
\end{lemma}
\begin{proof}
Let $m\in \mathbb{Z}^+$. Since $\mathbb{S}^{m-1}(n)$ is defined by \eqref{definition_Sm_n}, we will have that for any 
$\mathbf{H}\in\mathbb{S}^{m-1}(n)$, $\Lambda(\mathbf{H})\subseteq \mathbb{S}^{m-1}\subseteq\Pi_{j=1}^{m} [-1,1]\subseteq \mathbb{R}^{m+1}$. It can be seen that the identity mapping $\mathrm{id}_{M_n^{m}}$ induces a one to one correspondence between $\mathbb{S}^{m-1}(n)$ and $V^{m}(\mathbb{S}^{m-1})(n)$. Because of this correspondence, we have in fact that $\mathbb{S}^{m-1}(n)=V^{m}(\mathbb{S}^{m-1})(n)$. Since by Theorem \ref{main_generic_result} we have that $V^{m}(\mathbb{S}^{m-1})(n)$ is {\bf ULPDC}, this implies that $\mathbb{S}^{m-1}(n)=V^{m}(\mathbb{S}^{m-1}(n)$ is {\bf ULPDC}. This completes the proof.
\end{proof}

\subsection{Local Connectivity of Commutative $C^\ast$-Subalgebras of Matrix Algebras}
\label{C_star_matrix_connectivity}

Let us consider the universal $C^*$-algebras $C([-1,1]^m)$, $C(\TT[m])$ and $C(\TT\times [-1,1])$ described in terms of generators and relations by the expressions.

\begin{eqnarray}
C([-1,1]^m)&:=&
 C^*_1\left\langle h_1,\ldots,h_m \left| 
                     \begin{array}{l}
                        [h_j,h_k]=0\\
                       -1\leq h_j\leq 1,
                      \end{array}, 1\leq j,k\leq m
\right.\right\rangle,\nonumber\\
\label{universal_picture_of_Bm}\\
C(\TT[m])&:=&
 C^*_1\left\langle u_1,\ldots,u_m \left| 
                     \begin{array}{l}
                       \:[u_j,u_k]=0, \\
                       u_ju_j^\ast=u_j^\ast u_j= 1, 
                      \end{array}, 1\leq j,k\leq m
\right.\right\rangle,\nonumber\\
\label{universal_picture_of_Tm}\\
C(\mathbb{S}^{m-1})&:=&
 C^*_1\left\langle h_1,\ldots,h_m \left| 
                     \begin{array}{l}
                        [h_j,h_k]=0=h_j-h_j^\ast, \:\: 1\leq j,k\leq m\\
                       \sum_{j=1}^m h_j^2=1
                      \end{array}
\right.\right\rangle.\nonumber\\
\label{universal_picture_of_Sm}
\end{eqnarray}

\subsubsection{{\bf ULPDC} property for $C^\ast$-algebras}
\label{def.ulac} 

Given a universal $C^*$-algebra $A$ that can be described by the expression $A:=C^*_1\langle\NS{x}{m}, 
\mathcal{R}(x_1,\ldots,x_m)\rangle$, we say that 
$\mathrm{C^\ast Rep}(A,\Minf)$ (the set of all $C^*$-representations of $A$ in $\Minf:=\overline{\bigcup_{n\in\ZZ^+}M_n}^{\|\cdot\|}$) is uniformly locally piecewise differentiably connected or 
({\bf ULPDC}) if for 
any two $m$-tuples of normal contractions $\mathbf{X}:=(\NS{X}{m})$ and $\mathbf{Y}:=(Y_1,\ldots,Y_m)$ in 
 $M_n^m$ such that there exist $\ast$-representations $\pi_X\in \mathrm{C^\ast Rep}(A,C^\ast(\mathbf{X}))$ and $\pi_Y\in\mathrm{C^\ast Rep}(A,C^\ast(\mathbf{Y}))$ determined by the equations $\pi_X(x_j)=X_j$ and $\pi_Y(x_j)= Y_j$, $1\leq j\leq m$, we have that
 there are piecewise differentiable contractive normal matrix paths $\hat{X}^j$ that solve the problems $X_j \rightsquigarrow_{\varepsilon(\delta)} Y_j$, for some function $\varepsilon:\RR^+_0\to \mathbb{R}_0^+$, with 
 $\eth(\mathbf{X},\mathbf{Y})\leq \delta$, and if in addition there is $\pi_t\in \mathrm{C^\ast Rep}(A,C^*(\hat{X}^1(t),\ldots,\hat{X}^m(t))$ for each 
 $0\leq t\leq 1$.

\begin{theorem}
 For any integer $m\geq 1$ we have that $\Rep{C[-1,1]^m}{\Minf}$ is {\bf ULPDC}.
\end{theorem}
\begin{proof}
Since we have that $C([-1,1]^m)$ is described by \ref{universal_picture_of_Bm}. It can be seen that 
 for any $\rho_n\in \Rep{C([-1,1]^m)}{M_n}$ the $m$-tuple $(\rho_n(h_1),\ldots,\rho_n(h_m))$ belongs to $I^m(n)$. 
 By Theorem \ref{main_result}, for any given $\varepsilon>0$, there is $\delta> 0$ such that 
 for any two elements $\rho_n,\rho'_n\in\Rep{C([-1,1]^m)}{M_n}$ that satisfy the constraint
 \[
 \eth((\rho_n(h_1),\ldots,\rho_n(h_m)),(\rho'_n(h_1),\ldots,\rho'_n(h_m)))\leq \delta,
 \] 
 there is a piecewise differentiable path $\gamma\in C([0,1],\mathbb{I}^m(n))$ that solves the problem 
 \[
 (\rho_n(h_1),\ldots,\rho_n(h_m)) \rightsquigarrow_{\varepsilon} (\rho'_n(h_1),\ldots,\rho'_n(h_m))
 \]
relative to $\mathbb{I}^m(n)\cap B_\eth(\rho_n(h_1),\ldots,\rho_n(h_m),\varepsilon)$. By universality of $C([-1,1]^m)$ we will have that there is a $\ast$-representation
\[
\pi_t:C([-1,1]^m) \to C^*(\gamma(t)[1],\ldots,\gamma(t)[m])\subseteq M_n
\]
for each $0\leq t\leq 1$. This completes the proof.
\end{proof}

\begin{theorem}
 For any integer $m\geq 1$ we have that $\Rep{C(\TT[m])}{\Minf}$ is {\bf ULPDC}.
\end{theorem}
\begin{proof}
Since we have that $C(\TT[m])$ is described by \ref{universal_picture_of_Tm}. It can be seen that 
 for any $\rho_n\in \Rep{C(\TT[m])}{M_n}$ the $m$-tuple $(\rho_n(u_1),\ldots,\rho_n(u_m))$ belongs to $\mathbb{T}^m(n)$. 
 By Lemma \ref{secondary_corollary}, for any given $\varepsilon>0$, there is $\delta> 0$ such that 
 for any two elements $\rho_n,\rho'_n\in\Rep{C(\TT[m])}{M_n}$ that satisfy the constraint
 \[
 \eth((\rho_n(u_1),\ldots,\rho_n(u_m)),(\rho'_n(u_1),\ldots,\rho'_n(u_m)))\leq \delta,
 \] 
 there is a piecewise differentiable path $\gamma\in C([0,1],\mathbb{T}^m(n))$ that solves the problem
 \[
  (\rho_n(u_1),\ldots,\rho_n(u_m)) \rightsquigarrow_{\varepsilon} (\rho'_n(u_1),\ldots,\rho'_n(u_m)),
 \]
relative to $\mathbb{T}^m(n)\cap B_\eth(\rho_n(u_1),\ldots,\rho_n(u_m),\varepsilon)$. By universality of $C(\TT[m])$ we will have that there is a $\ast$-representation
\[
\pi_t:C(\TT[m]) \to C^*(\gamma(t)[1],\ldots,\gamma(t)[m])\subseteq M_n
\]
for each $0\leq t\leq 1$. This completes the proof.
\end{proof}

\begin{theorem}
For any integer $m\geq 1$ we have that $\Rep{C(\mathbb{S}^{m-1})}{\Minf}$ is {\bf ULPDC}.
\end{theorem}
\begin{proof}
Since we have that $C(\mathbb{S}^{m-1})$ is described by \ref{universal_picture_of_Sm}. It can be seen that 
 for any $\rho_n\in \Rep{C(\mathbb{S}^{m-1})}{M_n}$ the $m$-tuple $(\rho_n(h_1),\ldots,\rho_n(h_{m}))$ belongs to $\mathbb{S}^{m-1}(n)$. 
 By Lemma \ref{third_corollary}, for any given $\varepsilon>0$, there is $\delta> 0$ such that 
 for any two elements $\rho_n,\rho'_n\in\Rep{C(\mathbb{S}^{m-1})}{M_n}$ that satisfy the constraint
 \[
 \eth((\rho_n(h_1),\ldots,\rho_n(h_m)),(\rho'_n(h_1),\ldots,\rho'_n(h_m)))\leq \delta,
 \] 
 there is a piecewise differentiable path $\gamma\in C([0,1],\mathbb{S}^{m-1}(n))$ that solves the problem
 \[
  (\rho_n(h_1),\ldots,\rho_n(h_m)) \rightsquigarrow_{\varepsilon} (\rho'_n(h_1),\ldots,\rho'_n(h_m)),
 \]
relative to $\mathbb{S}^{m-1}(n)\cap B_\eth(\rho_n(h_1),\ldots,\rho_n(h_m),\varepsilon)$. By universality of $C(\mathbb{S}^{m-1})$ we will have that there is a $\ast$-representation
\[
\pi_t:C(\mathbb{S}^{m-1}) \to C^*(\gamma(t)[1],\ldots,\gamma(t)[m])\subseteq M_n
\]
for each $0\leq t\leq 1$. This completes the proof.
\end{proof}

\subsubsection{Spherical Unitaries and Completely Positive Maps on Matrix Algebras} Given $m,n\in \mathbb{Z}^+$, let us consider any unital completely positive ({\bf CP}) map (in the sense of \cite{CP_maps_Choi} and \cite{CP_fixed_points}) $\psi_{\mathbf{S}}:\Minf\to \Minf$ determined for each $X\in M_n$ by some $m$-tuple $\mathbf{S}=(S_1,\ldots,S_m)\in M_n^m$ by the expression
\begin{equation}
\psi_{\mathbf{S}}(X)=\sum_{k=1}^m S_jXS_j^\ast,
\label{def_spherical_cp_map}
\end{equation}
We will call the map $\psi_{\mathbf{S}}$, the CP map induced by $\mathbf{S}$. If in addition we have that $\mathbf{S}\in \mathbb{D}^m(n)$, and if the elements of $\mathbf{S}$ satisfy the contraint $\sum_{j=1}^m S_j S_j^\ast=\mathbf{1}_n$, we will call a map like $\psi_{\mathbf{S}}$ a $m-$spherical CP ($m$-SCP) map, since the $m$-tuple $\mathbf{S}$ is a matrix representation of the universal operators known as spherical unitaries in the sense of \cite{Davidson_1} and \cite{CP_fixed_points}. The set of $m$-SCP maps acting on $M_n$ will be denoted by $\mathrm{SCP}_m(M_n)$, similarly the set of $m$-SCP maps acting on $\Minf$ will be denoted by $\mathrm{SCP}_m(\Minf)$ . We will call any $m$-tuple like $\mathbf{S}$ a spherical unitary, and will denote by $\mathbb{SU}^m(n)$ the set of all spherical unitaries in $M_n^m$ described as follows. 
\begin{equation}
\mathbb{SU}^{m}(n)=\left\{\mathbf{S}\in \mathbb{D}^{m}(n)\:\left|\: \sum_{k=1}^m S_kS_k^\ast=\mathbf{1}_n\right.\right\}
\label{SU_m_rep}
\end{equation}

\begin{lemma}
\label{spherical_unitary_connectivity}
The matrix semialgebraic variety $\mathbb{SU}^{m}(n)$ is {\bf ULPDC}.
\end{lemma}
\begin{proof}
Let $m\in \mathbb{Z}^+$ and let us consider the maps $\pi_H$ and $\upsilon_H$ implemented in the proof of 
Lemma \ref{secondary_corollary}. Since $\mathbb{SU}^m(n)$ is defined by \eqref{SU_m_rep}, we will have that for any $\mathbf{S}\in\mathbb{SU}^m(n)$ we have that $\Lambda(\pi_H(\mathbf{S}))\subseteq \mathbb{S}^{2m-1}\subseteq \Pi_{j=1}^{2m} [-1,1]\subseteq \mathbb{R}^{2m}$. As in 
the proof of Lemma \ref{secondary_corollary} we will have that there is a one to one correspondence between $\mathbb{SU}^m(n)$ and $V^{2m}(\mathbb{S}^{2m-1})(n)$. Because of this correspondence, we have that any $\mathbf{S}\in \mathbb{SU}^m(n)$ has a representation of the form $\mathbf{S}=\upsilon_H(\pi_H(\mathbf{S}))$. 

Given  any two $(\nu,\eth)$-close tuples $\mathbf{S},\mathbf{T}\in \mathbb{SU}^m(n)$ for some given $\nu>0$, we will have that the tuples $\pi_H(\mathbf{S}),\pi_H(\mathbf{T})\in V^{2m}[\mathbb{S}^{2m-1}](n)$ will be $(\nu,\eth)$-close. By Theorem \ref{main_generic_result} we have that $V^{2m}(\mathbb{S}^{2m-1})(n)$ is {\bf ULPDC}, this implies that given $\varepsilon>0$, there is $\delta>0$ such that if $\nu<\delta$, there is a piecewise differentiable path $\gamma\in C([0,1],V^{2m}[\mathbb{S}^{2m-1}](n))$ that solves the connectivity problem $\pi_H(\mathbf{S})\rightsquigarrow_{\varepsilon/2} \pi_H(\mathbf{T})$ relative to $V^{2m}[\mathbb{S}^{2m-1}](n)\cap B_\eth(\pi_H(\mathbf{U}),\varepsilon/2)$. By the previous facts we will have that the path $\hat{\gamma}$ defined by the equation $\hat{\gamma}(t)=\upsilon_H(\gamma(t))$ for each $t\in [0,1]$, solves the connectivity problem $\mathbf{S}\rightsquigarrow_\varepsilon \mathbf{T}$ relative to $\mathbb{SU}^m(n)\cap B_\eth(\mathbf{U},\varepsilon)$. This completes the proof.
\end{proof}

\begin{theorem}
\label{universal_spherical_unitary_connectivity}
For any integer $m\geq 1$ and any universal spherical unitary $\mathbf{S}=(s_1,\ldots,s_m)$ we have that $\Rep{C^\ast(\mathbf{S})}{\Minf}$ is {\bf ULPDC}.
\end{theorem}
\begin{proof}
Since we have that $C^\ast(\mathbf{S})$ is described in terms of generators and relations by the expression. 
\begin{equation}
C^\ast(\mathbf{S}):=
 C^*_1\left\langle s_1,\ldots,s_m \left| 
                     \begin{array}{l}
                        [s_j,s_k]=0=s_js_j^\ast-s_j^\ast s_j, \:\: 1\leq j,k\leq m\\
                       \sum_{j=1}^m s_js_j^\ast=1
                      \end{array}
\right.\right\rangle.
\label{uinversal_SU_rep}
\end{equation}
It can be seen that 
 for any $\rho_n\in \Rep{C^\ast(\mathbf{S})}{M_n}$ the $m$-tuple $(\rho_n(s_1),\ldots,\rho_n(s_{m}))$ belongs to $\mathbb{SU}^{m}(n)$. 
 By Lemma \ref{spherical_unitary_connectivity}, for any given $\varepsilon>0$, there is $\delta> 0$ such that 
 for any two elements $\rho_n,\rho'_n\in\Rep{C^\ast(\mathbf{S})}{M_n}$ that satisfy the constraint
 \[
 \eth((\rho_n(s_1),\ldots,\rho_n(s_m)),(\rho'_n(s_1),\ldots,\rho'_n(s_m)))\leq \delta,
 \] 
 there is a piecewise differentiable path $\gamma\in C([0,1],\mathbb{SU}^{m}(n))$ that solves the problem
 \[
  (\rho_n(s_1),\ldots,\rho_n(s_m)) \rightsquigarrow_{\varepsilon} (\rho'_n(s_1),\ldots,\rho'_n(s_m)),
 \]
relative to $\mathbb{SU}^{m}(n)\cap B_\eth(\rho_n(s_1),\ldots,\rho_n(s_m),\varepsilon)$. By universality of $C^\ast(\mathbf{S})$ we will have that there is a $\ast$-representation
\[
\pi_t:C^\ast(\mathbf{S}) \to C^*(\gamma(t)[1],\ldots,\gamma(t)[m])\subseteq M_n
\]
for each $0\leq t\leq 1$. This completes the proof.
\end{proof}

Let us write $\mathrm{CP}(M_n)$ to denote the set of completely positive linear maps on $M_n$, let us extend the {\bf ULPDC} property to subsets of $\mathrm{CP}(M_n)$.

\begin{definition} 
We say that s subset $\mathcal{S}\subseteq \mathrm{CP}(M_n)$ is {\bf ULPDC} if given two maps $\psi_{\mathbf{S}},\psi_{\mathbf{T}}\in \mathcal{S}$ induced by two $m$-tuples $\mathbf{S},\mathbf{T}\in M_n^m$ and given $\varepsilon>0$, there is $\delta>0$ such that if $\eth(\mathbf{S},\mathbf{T})< \delta$, then there is a piecewise differentiable path $\{\Psi_t\}_{t\in[0,1]}\subseteq \mathcal{S}$ such that $\Psi_0=\psi_{\mathbf{S}}$, $\Psi_1=\psi_{\mathbf{T}}$, and $\|\Psi_{t}(X)-\Psi_{\mathbf{S}}(X)\|<\varepsilon$ for each $t\in [0,1]$ and any contraction $X\in M_n$.
\end{definition}

The following lemma was motivated by some connections between the connectivity properties of $\mathrm{CP}(\Minf)$ and the classification of quantum phases in the sense of \cite{connectivity_CP_maps_q_phases}.

\begin{lemma}
\label{SCP_maps_connectivity}
For any integer $m\geq 1$ we have that $\mathrm{SCP}_m(\Minf)$ is {\bf ULPDC}.
\end{lemma}
\begin{proof}
Let $m\in \mathbb{Z}^+$. It can be seen that for any $\mathbf{S},\mathbf{T}\in \mathbb{SU}^m(n)$ and any $X\in M_n$ such that $\|X\|\leq 1$, the following estimate holds for the corresponding induced CP maps.
\begin{eqnarray}
\|\psi_{\mathbf{S}}(X)-\psi_{\mathbf{T}}(X)\|&=&\left\|\sum_{j=1}^m S_jXS_j^\ast-\sum_{j=1}^mT_jXT_j^\ast\right\|\nonumber\\
											 &\leq&\sum_{j=1}^m\|S_jXS_j^\ast-T_jXT_j^\ast\|\nonumber\\
                                             &\leq&\sum_{j=1}^m2\|S_j-T_j\|\|X\|\nonumber\\
                                             &\leq&2m\eth(\mathbf{S},\mathbf{T}).
                                             \label{first_SCP_bound_estimate}
\end{eqnarray}
Since $\mathbb{SU}^m(n)$ is {\bf ULPDC} by Lemma \ref{spherical_unitary_connectivity}, we have that given $\varepsilon>0$, there is $\delta>0$ such that for any two $m$-tuples $\mathbf{S},\mathbf{T}\in \mathbb{SU}^m(n)$ that satisfy the constraint $\eth(\mathbf{S},\mathbf{T})<\delta$, there is a piecewise differentiable path $\gamma\in C([0,1],\mathbb{SU}^m(n))$ that solves the problem $\mathbf{S}\rightsquigarrow_{\varepsilon/(2m)} \mathbf{T}$ relative to $\mathbb{SU}^m(n)\cap B_\eth(\mathbf{S},\varepsilon/(2m))$. Let us set $\Psi_t=\psi_{\gamma(t)}$ for each $t\in [0,1]$. It can be seen that 
$\Psi_0=\psi_{\mathbf{S}}$, $\Psi_1=\psi_{\mathbf{T}}$, $\Psi_t\in \mathrm{SCP}_m(M_n)$ for each $t\in [0,1]$, and by the estimate \eqref{first_SCP_bound_estimate} we will have that for any $X\in M_n$ such that $\|X\|\leq 1$,
\[
\|\Psi_t(X)-\psi_{\mathbf{S}}(X)\|\leq 2m\eth(\gamma(t),\mathbf{S})< 2m(\varepsilon/(2m))=\varepsilon.
\]
This completes the proof.
\end{proof}

\section{Hints and Future Directions}
\label{hints}
The implications of the results in sections \S\ref{main_results} and \S\ref{applications} in uniform local path connectivity of other subsets of $CP(\Minf)$ will be further explored, some connections to fixed pont approximations for particular types of elements in $\mathrm{CP}(\Minf)$ will be studied as well. 

By considering matrix paths as continuous/differentiable analogies of numerical linear algebra algorithms in the sense of \cite{Chu_num_lin}, we will work on the adaptation and application of the results in sections \S\ref{main_results} and \S\ref{applications} to the structure-preserving approximation/perturbation of families of structured matrices with some particular spectral behavior in the sense of 
\cite{Spectral_refinement_1,Spectral_refinement_2,ASD_matrices,sim_block_diag,
Clustered_matrix_approximation}.

\section*{Acknowledgement}
I am very grateful with the Fields Institute for Research in Mathematical Sciences at the University of Toronto, for their support during my research visit to the Institute from July 31 to August 4, 2017. Much of the research reported in this document was carried out while I was visiting the Institute.

I am grateful with Terry Loring, Alexandru Buium, Alexandru Chirvasitu, Moody Chu, Marc Rieffel, Stanly Steinberg, Yossi Avron, Masoud Khalkhali and Concepci\'on Ferrufino, for several interesting questions and comments that have been very helpful for the preparation of this document.

\end{document}